\documentclass[reqno]{amsart}

\usepackage{amsmath}
\usepackage{amsfonts}
\usepackage{pdfsync}
\usepackage{color}
\usepackage{graphicx}
\usepackage{amssymb}
\usepackage{epsf}
\usepackage{eucal}
\usepackage{amscd}
\usepackage{epsfig}
\usepackage{graphics}
\usepackage{amsthm}
\usepackage{fancyhdr}
\usepackage{delarray}
\usepackage{fullpage}
\usepackage{enumerate}
\usepackage{paralist}
\usepackage[all]{xy}
\usepackage{tikz}
\usepackage{verbatim}
\usepackage{cite}

\usepackage{hyperref}

\definecolor{darkgreen}{cmyk}{1,0,1,.2}
\definecolor{m}{rgb}{1,0.1,1}

\newdimen\theight
\def\TeXref#1{%
             \leavevmode\vadjust{\setbox0=\hbox{{\tt
                     \quad\quad  {\small \textrm #1}}}%
             \theight=\ht0
             \advance\theight by \lineskip
             \kern -\theight \vbox to
             \theight{\rightline{\rlap{\box0}}%
             \vss}%
             }}%

\newcommand{\sm}{\smallsetminus} 

\renewcommand{\epsilon}{\varepsilon}
\newcommand{\ol}{\overline}

\DeclareMathOperator{\Cl}{Cl}

\DeclareMathOperator{\dom}{dom}
\DeclareMathOperator{\im}{im}

\DeclareMathOperator{\id}{id} 
\DeclareMathOperator{\supp}{supp}

\DeclareMathOperator{\Iso}{Iso}

\DeclareMathOperator{\Int}{Int}
\DeclareMathOperator{\ev}{ev}

\newcommand{\FF}{\mathcal{F}}

\newcommand{\KK}{\mathcal{K}} 
\newcommand{\GG}{\mathcal{G}}
\newcommand{\CC}{\mathcal{C}}
\newcommand{\MM}{\mathcal{M}}

\newcommand{\HH}{\mathcal{H}}
\newcommand{\VV}{\mathcal{V}}
\newcommand{\UU}{\mathcal{U}}

\newcommand{\DD}{\mathcal{D}}
\newcommand{\WW}{\mathcal{W}}
\newcommand{\NN}{\mathcal{N}}
\newcommand{\ZZ}{\mathcal{Z}}

\newcommand{\Z}{\mathbb{Z}}

\newcommand{\E}{\mathbb{E}}
\newcommand{\F}{\mathbb{F}}
\newcommand{\R}{\mathbb{R}}
\newcommand{\N}{\mathbb{N}}

\theoremstyle{plain}

\newtheorem{thm}{Theorem}[section]
\newtheorem{lem}[thm]{Lemma}
\newtheorem{cor}[thm]{Corollary}
\newtheorem{prop}[thm]{Proposition}

\theoremstyle{definition}

\newtheorem{defn}[thm]{Definition}

\newtheorem{ex}[thm]{Example}

\newtheorem{quest}[thm]{Question}

\theoremstyle{remark}

\newtheorem{rem}{Remark}

\newtheorem{claim}{Claim}

\title{Bounded geometry and leaves}

\author[J.A. \'Alvarez L\'opez]{Jes\'us A. \'Alvarez L\'opez}
\address{Departamento de Xeometr\'{\i}a e Topolox\'{\i}a\\
         Facultade de Matem\'aticas\\
         Universidade de Santiago de Compostela\\
         Campus Vida\\
         15782 Santiago de Compostela\\
         Spain}
\email{jesus.alvarez@usc.es}
  
\author[R. Barral Lij\'o]{Ram\'on Barral Lij\'o}
\address{Departamento de Xeometr\'{\i}a e Topolox\'{\i}a\\
         Facultade de Matem\'aticas\\
         Universidade de Santiago de Compostela\\
         Campus Vida\\
         15782 Santiago de Compostela\\
         Spain}
\email{ramon.barral@usc.es}


\thanks{The first author is partially supported by MICINN (Spain), grant MTM2011-25656}

	
\date{}

\begin{document}

\maketitle

\begin{abstract}  
  The main theorem states that any complete connected Riemannian manifold of bounded geometry can be isometrically realized as a leaf with trivial holonomy in a compact Riemannian foliated space.
\end{abstract}  

\tableofcontents

\section{Introduction}\label{s: intro}

Recall that a foliated space $X\equiv(X,\FF)$ of dimension $n$ is a topological space $X$ equipped with a partition $\FF$ into connected manifolds (leaves) so that $X$ can be locally described as a product $B\times Z$, where $B$ is an open ball in $\R^n$ and $Z$ any topological space (local transversal), and the slices $B\times\{*\}$ correspond to open sets in the leaves. This $\FF$ is called a foliated structure or lamination. Foliated spaces are usually assumed to be Polish to get better properties. Many basic notions about foliations can be obviously extended to foliated spaces, like foliated charts, plaques, foliated atlas, holonomy pseudogroup, holonomy group and holonomy covering of the leaves, minimality, transitivity, foliated maps, etc. Some basic results can be extended as well; for instance, there is an obvious version of the Reeb local stability theorem, and the union of leaves without holonomy is a meager subset if $X$ is second countable. Interesting classes of foliated spaces show up in several areas of mathematics, like in dynamics, arithmetics, tessellations, graphs and foliation theory (minimal sets). 

A $C^\infty$ foliated structure is given by a foliated atlas whose changes of coordinates are leafwise $C^\infty$, with ambient-space-continuous leafwise derivatives of arbitrary order. This gives rise to the concept of $C^\infty$ foliated space. To emphasize the difference, the foliated structure underlying a $C^\infty$ foliated structure may be called topological. On a $C^\infty$ foliated space $X\equiv(X,\FF)$, the concept of $C^\infty$ function is defined by requiring that its local expressions, using foliated coordinates, are leafwise $C^\infty$, with ambient-space-continuous leafwise partial derivatives of arbitrary order. $C^\infty$ bundles and sections also make sense on $X$, defined by requiring that their local descriptions are given by $C^\infty$ functions in the above sense. For instance, the tangent bundle $TX$ (or $T\FF$) is the $C^\infty$ vector bundle on $X$ that consists of the vectors tangent to the leaves, and a Riemannian metric on $X$ consists of Riemannian metrics on the leaves fitting together nicely to form a $C^\infty$ section on $X$. This gives rise to the concept of Riemannian foliated space.

$C^\infty$ foliated maps between $C^\infty$ foliated spaces can be similarly defined; in particular, $C^\infty$ foliated immersions, submersions, (local) diffeomorphisms and (local) embeddings between $C^\infty$ foliated spaces have obvious meanings. If a homeomorphism between $C^\infty$ foliated spaces is $C^\infty$ and its restrictions to the leaves are diffeomorphisms, then it is a $C^\infty$ diffeomorphism, as follows easily from the continuity of the inversion of $C^\infty$ diffeomorphisms between $C^\infty$ manifolds with respect to the $C^\infty$ topology \cite[p.~64, Exercise~9]{Hirsch1976}. Several results about foliated spaces have obvious $C^\infty$ versions, like the Reeb local stability theorem.

Standard references about foliated spaces are \cite{MooreSchochet1988}, \cite[Chapter~11]{CandelConlon2000-I}, \cite[Part~1]{CandelConlon2003-II} and \cite{Ghys2000}. See also \cite[Section~2.1]{AlvarezBarralCandel2016} for a quick summary of what is needed here.

On the other hand, recall that a Riemannian manifold $M$ is said to be of bounded geometry when it has a positive injectivity radius, and the $m$-th covariant derivative of the curvature tensor has uniformly bounded norm for all order $m$; in particular, $M$ is complete by the positivity of the injectivity radius. The following are typical examples where bounded geometry holds: coverings of closed connected Riemannian manifolds, connected Lie groups with left invariant metrics, and leaves of compact Riemannian foliated spaces. More examples can be produced by using compactly supported perturbations of given Riemannian manifolds of bounded geometry. In fact, any smooth manifold admits a metric of bounded geometry \cite{Greene1978}. We will focus in the case of leaves of compact Riemannian foliated spaces, showing that this example indeed characterizes bounded geometry.

\begin{thm}\label{t: bounded geometry => leaf of a compact fol sp}
	Any connected Riemannian manifold of bounded geometry is isometric to a leaf with trivial holonomy of some compact Riemannian foliated space.
\end{thm}

It is commonly accepted that such a result should be true, and that it should follow by using the closure of the canonical embedding of the manifold into the Gromov space $\MM_*$ of pointed proper metric spaces\cite{Gromov1981}, \cite[Chapter~3]{Gromov1999}, or, better, into its smooth version, the space $\MM_*^\infty(n)$ of isometry classes of pointed complete connected Riemannian $n$-manifolds with the topology defined by the $C^\infty$ convergence \cite[Chapter~10, Section~3.2]{Petersen1998}, \cite[Theorem~1.2]{AlvarezBarralCandel2016}. However, to the authors knowledge, no complete proof has been given so far.

A complete connected Riemannian $n$-manifold $M$ is called non-periodic (respectively, locally non-periodic) if $\Iso(M)=\{\id_M\}$ (respectively, the canonical projection $M\to\Iso(M)\backslash M$ is a covering map), where $\Iso(M)$ denotes the isometry group of $M$. The non-periodic and locally non-periodic manifolds define subspaces of $\MM_*^\infty(n)$ respectively denoted by $\MM_{*,\text{\rm np}}^\infty(n)$ and $\MM_{*,\text{\rm lnp}}^\infty(n)$. There is a canonical map $\iota_M:M\to\MM_*^\infty(n)$, given by $\iota_M(x)=[M,x]$ (the isometry class of $(M,x)$), which induces a continuous injection $\bar{\iota}_M:\Iso(M)\backslash M\to\MM_*^\infty(n)$. The images of all possible maps $\iota_M$ form a partition $\FF_*(n)$ of $\MM_*^\infty(n)$. The restriction of $\FF_*(n)$ to $\MM_{*,\text{\rm lnp}}^\infty(n)$ is denoted by $\FF_{*,\text{\rm lnp}}(n)$. For $n\ge2$, $\MM_{*,\text{\rm lnp}}^\infty(n)$ is open and dense in $\MM_*^\infty(n)$, and $\FF_{*,\text{\rm lnp}}(n)$ is a Riemannian foliated space of dimension $n$ so that each map $\iota_M:M\to\im\iota_M$ is a local isometry and the holonomy covering of the leaf $\im\iota_M$ \cite[Theorem~1.3]{AlvarezBarralCandel2016}; in particular, $\MM_{*,\text{\rm np}}^\infty(n)$ is the union of leaves with trivial holonomy. Moreover $\Cl_\infty(\im\iota_M)$ is compact  if and only if $M$ is of bounded geometry \cite[Theorem~12.3]{AlvarezBarralCandel2016} (see also \cite{Cheeger1970}, \cite[Chapter~10, Sections~3 and~4]{Petersen1998}), where $\Cl_\infty$ denotes the closure operator in $\MM_*^\infty(n)$. Then, analyzing the cases where $\Cl_\infty(\im\iota_M)\subset\MM_{*,\text{\rm lnp}}(n)$, a version of Theorem~\ref{t: bounded geometry => leaf of a compact fol sp} follows assuming restrictions on $M$ \cite[Theorem~1.5]{AlvarezBarralCandel2016}. 

To prove Theorem~\ref{t: bounded geometry => leaf of a compact fol sp} with complete generality, we refine the above arguments as follows. Fix a separable Hilbert space $\E$ and any natural\footnote{It is assumed that $0$ is natural.} $n$. Consider pairs $(M,f)$ and triples $(M,f,x)$, where $M$ is a complete connected Riemannian $n$-manifold, $f\in C^\infty(M,\E)$ and $x\in M$. An \emph{equivalence} $\phi:(M,f)\to(N,h)$ is an isometry $\phi:M\to N$ such that $\phi^*h=f$. If moreover distinguished points, $x\in M$ and $y\in N$, are preserved, then $\phi:(M,f,x)\to(N,h,y)$ is called a \emph{pointed equivalence}. The group of self equivalences of $(M,f)$ is denoted by $\Iso(M,f)$. If there is a pointed equivalence $(M,f,x)\to(N,h,y)$, then the triples $(M,f,x)$ and $(N,h,y)$ are declared to be \emph{equivalent}. The equivalence class of each $(M,f,x)$ is denoted by $[M,f,x]$. Let $\widehat\MM_*(n)$ denote the set\footnote{Like in the cases of $\MM_*$ and $\MM_*^\infty(n)$, without loss of generality, it can be assumed that the underlying set of any such $M$ is contained in \(\R\), so that $\widehat\MM_*(n)$ becomes a well defined set.} of such equivalence classes.
  
\begin{defn}\label{d: C^infty convergence in widehat MM_*(n)}
	For each $m\in\N$, a sequence $[M_i,f_i,x_i]$ in $\widehat\MM_*(n)$ is said to be \emph{$C^m$ convergent} to $[M,f,x]\in\widehat\MM_*(n)$ if, for each compact domain\footnote{Here, a \emph{domain} in $M$ is a connected $C^\infty$ submanifod, possibly with boundary, of the same dimension as $M$.} $\Omega\subset M$ containing $x$, there is a pointed $C^{m+1}$ embedding $\phi_i:(\Omega,x)\to(M_i,x_i)$ for each large enough $i$ such that $\phi_i^*g_i\to g|_\Omega$ and $\phi_i^*f_i\to f|_\Omega$ as $i\to\infty$ with respect to the $C^m$ topology \cite[Chapter~2]{Hirsch1976}. If $[M_i,f_i,x_i]$ is $C^m$ convergent to $[M,f,x]$ for all $m$, then it is said that $[M_i,f_i,x_i]$ is \emph{$C^\infty$ convergent} to $[M,f,x]$.
\end{defn}

It is not completely obvious that this $C^\infty$ convergence satisfies the conditions to define a topology \cite{Koutnik1985}, \cite{GutierresHofmann2007}. Thus the following result is not trivial.

\begin{thm}\label{t: C^infty convergence in widehat MM_*(n)} 
	The $C^\infty$ convergence in $\widehat\MM_*(n)$ describes a Polish topology.
\end{thm}

The topology given by Theorem~\ref{t: C^infty convergence in widehat MM_*(n)} will be called the \emph{$C^\infty$ topology}, and the corresponding space is denoted by $\widehat\MM_*^\infty(n)$. The closure operator in this space will be denoted by $\widehat{\Cl}_\infty$. The following maps are canonical and continuous: a \emph{forgetful} map $\widehat\MM_*^\infty(n)\to\MM_*^\infty(n)$, $[M,f,x]\mapsto[M,x]$, and an \emph{evaluation} map $\ev:\widehat\MM_*^\infty(n)\to\E$, $[M,f,x]\mapsto f(x)$. Note that $\ev:\widehat\MM_*(0)\to\E$ is a homeomorphism. Moreover, for each complete connected Riemannian $n$-manifold $M$ and any $f\in C^\infty(M,\E)$, there is a canonical continuous map $\hat\iota_{M,f}:M\to\widehat\MM_*^\infty(n)$, given by $\hat\iota_{M,f}(x)=[M,f,x]$, which induces a continuous injection $\bar{\iota}_{M,f}:\Iso(M,f)\backslash M\to\MM_*^\infty(n)$. The images of the maps $\hat\iota_{M,f}$ form a natural partition of $\widehat\MM_*^\infty(n)$, denoted by $\widehat\FF_*(n)$. Let $C^\infty_{\text{\rm imm}}(M,\E)$ be the set of $C^\infty$ immersions $M\to\E$, and let $\widehat\MM_{*,\text{\rm imm}}^\infty(n)$ be the $\widehat\FF_*(n)$-saturated subspace of $\widehat\MM_*(n)$ consisting of classes $[M,f,x]$ with $f\in C^\infty_{\text{\rm imm}}(M,\E)$. The restriction of $\widehat\FF_*(n)$ to $\widehat\MM_{*,\text{\rm imm}}(n)$ is denoted by $\widehat\FF_{*,\text{\rm imm}}(n)$. Observe that the canonical projection $M\to\Iso(M,f)\backslash M$ is a covering map if $f\in C^\infty_{\text{\rm imm}}(M,\E)$. 

On the other hand, let $\widehat\MM_{*,\text{\rm c}}^\infty(n)$ (respectively, $\widehat\MM_{*,\text{\rm o}}^\infty(n)$) be the $\widehat\FF_*(n)$-saturated subspace of $\widehat\MM_*(n)$ consisting of classes $[M,f,x]$ such that $M$ is compact (respectively, open). Observe that, if $[N,h,y]$ is close enough to any $[M,f,x]\in\widehat\MM_{*,\text{\rm c}}^\infty(n)$, then $N$ is diffeomorphic to $M$. Thus $\widehat\MM_{*,\text{\rm c}}^\infty(n)$ is open in $\widehat\MM_*(n)$, and therefore $\widehat\MM_{*,\text{\rm o}}^\infty(n)$ is closed. Hence these are Polish subspaces of $\widehat\MM_*(n)$, as well as their intersections with any Polish subspace. Let $\widehat\MM_{*,\text{\rm imm,c/o}}^\infty(n)=\widehat\MM_{*,\text{\rm c/o}}^\infty(n)\cap\widehat\MM_{*,\text{\rm imm}}^\infty(n)$. The restrictions of $\widehat\FF_*(n)$ to $\widehat\MM_{*,\text{\rm c/o}}(n)$ and $\widehat\MM_{*,\text{\rm imm,c/o}}(n)$ are denoted by $\widehat\FF_{*,\text{\rm c/o}}(n)$ and $\widehat\FF_{*,\text{\rm imm,c/o}}(n)$, respectively.

\begin{thm}\label{t: widehat FF_*,imm(n)}
  	The following properties hold:
  		\begin{enumerate}[{\rm (}i\/{\rm )}]

  			\item\label{i: widehat MM_*,imm(n) is G_delta} $\widehat\MM_{*,\text{\rm imm}}^\infty(n)$ is Polish and dense in $\widehat\MM_*^\infty(n)$.

  			\item\label{i: widehat FF_*,imm(n) is a foliated structure} $\widehat\FF_{*,\text{\rm imm}}(n)$ is a foliated structure of dimension $n$.
			
			\item\label{i: widehat FF_*,imm,o(n) is transitive} $\widehat\FF_{*,\text{\rm imm,o}}(n)$ is transitive.

  			\item\label{i: widehat FF_*,imm(n) is C^infty} There is a unique $C^\infty$ foliated structure $\widehat\FF_{*,\text{\rm imm}}^\infty(n)$ on $\widehat\MM_{*,\text{\rm imm}}^\infty(n)$, whose underlying topological foliated structure is $\widehat\FF_{*,\text{\rm imm}}(n)$, such that $\ev:\widehat\MM_{*,\text{\rm imm}}^\infty(n)\to\E$ is a $C^\infty$ immersion. 
			
			\item\label{i: widehat FF_*,imm(n) is Riemannian} There is a unique Riemannian metric on $\widehat\MM_{*,\text{\rm imm}}^\infty(n)\equiv(\widehat\MM_{*,\text{\rm imm}}^\infty(n),\widehat\FF_{*,\text{\rm imm}}^\infty(n))$ such that $\iota_{M,f}:M\to\hat\iota_{M,f}$ is a local isometry for all complete connected Riemannian $n$-manifold $M$ and $f\in C^\infty_{\text{\rm imm}}(M,\E)$.

  			\item\label{i: holonomy covering} For all $M$ and $f$ as above, the map $\hat\iota_{M,f}:M\to\im\hat\iota_{M,f}$ is the holonomy covering of the leaf $\im\hat\iota_{M,f}$.

  		\end{enumerate}
\end{thm}

It is possible to give a version of Theorem~\ref{t: widehat FF_*,imm(n)} closer to \cite[Theorem~1.3]{AlvarezBarralCandel2016}, using the subspace $\widehat\MM_{*,\text{\rm lnp}}^\infty(n)$ consisting of the classes $[M,f,x]$ such that $M\to\Iso(M,f)\backslash M$ is a covering map. Such a result could be proved with the obvious adaptation of the proof of \cite[Theorem~1.3]{AlvarezBarralCandel2016}, using the exponential map to define foliated charts. Instead, we have opted for studying $\widehat\MM_{*,\text{\rm imm}}^\infty(n)$ because, in this case, the immersions $f$ directly provide foliated charts.

The following result states that $\widehat\MM_{*,\text{\rm imm}}^\infty(n)$ is universal among the class of Polish Riemannian foliated spaces that satisfy a condition called covering-continuity (Definition~\ref{d: covering-cont}).

\begin{thm}\label{t: universal} 
	A Polish Riemannian foliated space $X$ of dimension $n$ with complete leaves is isometric to a saturated Riemannian foliated subspace of $\widehat\MM_{*,\text{\rm imm}}^\infty(n)$ if and only if $X$ is covering-continuous.
\end{thm}

In  Theorem~\ref{t: universal}, when $X$ consists of a single leaf $M$, the isometric injection of $M$ into $\widehat\MM_{*,\text{\rm imm}}^\infty(n)$ is $\hat\iota_{M,f}$ for any $C^\infty$ embedding $f:M\to\E$. If moreover $M$ is of bounded geometry, then $f$ can be chosen so that $\widehat{\Cl}_\infty(\im\hat\iota_{M,f})$ is a compact Riemannian foliated subspace of $\widehat\MM_{*,\text{\rm imm}}^\infty(n)$ (Proposition~\ref{p: bounded geometry}). Then Theorem~\ref{t: bounded geometry => leaf of a compact fol sp} follows by considering the isometric injection $\hat\iota_{M,f}:M\to\widehat{\Cl}_\infty(\im\hat\iota_{M,f})$.

There are examples of Lie groups with left invariant metrics that are not coarsely quasi-isometric to any finitely generated group \cite{ChaluleauPittet2001}, \cite{EskinFisherWhyte2012}. Applying the above argument to those Riemannian manifolds, we get compact Riemannian foliated spaces whose leaf holonomy covers are not coarsely quasi-isometric to any finitely generated group.

Theorem~\ref{t: bounded geometry => leaf of a compact fol sp} contrasts with the examples of connected Riemannian manifolds of bounded geometry whose quasi-isometry type cannot be realized as leaves of foliations of codimension one on closed manifolds \cite{AttieHurder1996}, \cite{Zeghib1994}, \cite{Schweitzer1995}, \cite{Schweitzer2011}. If the metric is not considered, any surface can be realized as a leaf of a codimension one foliation on a closed manifold \cite{CantwellConlon1987}, but this fails in higher dimension \cite{Ghys1985}, \cite{InabaNishimoriTakamuraTsuchiya1985}, \cite{AttieHurder1996}, \cite{Souza2011}, \cite{SchweitzerSouza2013}. The study of this realizability problem was initiated in \cite{Sondow1975}.

This work can be considered as a continuation of \cite{AlvarezBarralCandel2016}, and therefore many references to \cite{AlvarezBarralCandel2016} are included.

\section{Preliminaries}\label{s: prelim}

Let $M$ be a Riemannian manifold (possibly with boundary or corners). The following standard notation will be used. The metric tensor is denoted by $g$, the distance function on each of the connected components of \(M\) by $d$, the tangent bundle by $\pi:TM\to M$, the Levi-Civita connection by $\nabla$, and the open and closed balls of center $x\in M$ and radius $r>0$ by $B(x,r)$ and $\ol B(x,r)$, respectively. If needed, ``$M$'' will be added to all of the above notation as a subindex or superindex; when a family of Riemannian manifolds $M_i$ is considered, we may add the subindex or superindex ``$i$'' instead of ``$M_i$''. A covering space of $M$ is assumed to be equipped with the lift of $g$.

For $m\in\Z^+$, let $T^{(m)}M=T\cdots TM$ ($m$ times); we also set $T^{(0)}M=M$. If $l<m$, $T^{(l)}M$ is identified with a regular submanifold of $T^{(m)}M$ via zero sections, and therefore, for each $x\in M$, the notation $x$ may be also used for the zero elements of $T_xM$, $T_xTM$, etc. Let $\pi\colon T^{(m)}M \to T^{(l)}M$ be the vector bundle projection given by composing the tangent bundle projections; in particular, we have $\pi:T^{(m)}M \to M$. Given any $C^m$ map between Riemannian manifolds, $\phi:M\to N$, the induced map $T^{(m)}M\to T^{(m)}N$ will be denoted by $\phi_*^{(m)}$ (or simply $\phi_*$ if $m=1$).

Hilbert manifolds are also considered in some parts of the paper, using analogous notation.

The Levi-Civita connection determines a decomposition $T^{(2)}M=\HH\oplus\VV$, as direct sum of the horizontal and vertical subbundles. The \emph{Sasaki metric} on $TM$ is the unique Riemannian metric $g^{(1)}$ so that $\HH\perp\VV$ and the canonical identities $\HH_\xi\equiv T_\xi M \equiv \VV_\xi$ are isometries for every $\xi\in TM$ \cite{Sasaki1958}. Continuing by induction, for $m\ge2$, the \emph{Sasaki metric} on $T^{(m)}M$ is $g^{(m)}=(g^{(m-1)})^{(1)}$. The notation $d^{(m)}$ is used for the corresponding distance function on the connected components, and the corresponding open and closed balls of center $\xi\in T^{(m)}M$ and radius $r>0$ are denoted by $B^{(m)}(\xi,r)$ and $\ol B^{(m)}(\xi,r)$, respectively. We may add the subindex ``$M$'' to this notation if necessary, or the subindex ``$i$'' instead of ``$M_i$'' for a family of Riemannian manifolds $M_i$. From now on, $T^{(m)}M$ is assumed to be equipped with $g^{(m)}$. For $l<m$, $T^{(l)}M$ becomes a totally geodesic Riemannian submanifold of $T^{(m)}M$ orthogonal to the fibers of $\pi:T^{(m)}M\to T^{(l)}M$, which are also totally geodesic \cite[Remark~1-(i)--(iii)]{AlvarezBarralCandel2016} (see also \cite[Corollary of Theorem~13, and Theorems~14 and~18]{Sasaki1958}).

Let $(U;x^1,\dots,x^n)$ be a chart of $M$. As usual, the corresponding metric coefficients are denoted by $g_{ij}$, and write $(g^{ij})=(g_{ij})^{-1}$. Identify the functions $x^i$ with their lifts to $TU$. We get a chart $(U^{(1)};x_{(1)}^1,\dots,x_{(1)}^{2n})$ of $TM$ with $U^{(1)}=TU$, $x_{(1)}^i=x^i$ and $x_{(1)}^{n+i}=v^i$ for $1\le i\le n$, where the functions $v^i$ give the coordinates of tangent vectors with respect to the local frame $(\partial_1,\dots,\partial_n)$ of $TU$ induced by $(U;x^1,\dots,x^n)$. By induction, for $m\ge2$, let $(U^{(m)};x_{(m)}^1,\dots,x_{(m)}^{2^mn})$ be the chart of $T^{(m)}M$ induced by the chart $(U^{(m-1)};x_{(m-1)}^1,\dots,x_{(m-1)}^{2^{m-1}n})$ of $T^{(m-1)}M$.

Let $\Omega\subset M$ be a compact domain and $m\in\N$.  Fix a finite collection of charts of $M$ that covers $\Omega$, $\UU=\{(U_a;x_a^1,\dots,x_a^n)\}$, and a family of compact subsets of $M$ with the same index set as $\UU$, $\KK=\{K_a\}$, such that $\Omega\subset\bigcup_aK_a$, and $K_a\subset U_a$ for all $a$. The corresponding $C^m$ norm of a $C^m$ tensor $T$ on $\Omega$ is defined by\footnote{The standard multi-index notation is used here.}
	\[
		\|T\|_{C^m,\Omega,\UU,\KK} =\max_a\max_{x\in K_a\cap\Omega}
		\sum_{|I|\le m} \sum_{J,K}\left|\frac{\partial^{|I|}T_{a,J}^K}{\partial x_a^I}(x)\right|\;,
		\]
where $T_{a,J}^K$ are the coefficients of $T$ on $U_a\cap\Omega$ with respect to the frame induced by $(U_a;x_a^1,\dots,x_a^n)$. With this norm, the $C^m$ tensors on $\Omega$ of a fixed type form a Banach space, whose underlying topology is called the \emph{$C^m$ topology}. By taking the projective limit as $m\to\infty$, we get the Fr\'echet space of $C^\infty$ tensors of that type, whose underlying topology is called the \emph{$C^\infty$ topology} (see e.g.\ \cite{Hirsch1976}). We will always consider the $C^k$ topology for $C^k$ tensors on $\Omega$ of a given type ($k\in\N\cup\{\infty\}$); in particular, $C^k(\Omega)$ is always assumed to be equipped with the $C^k$ topology. Observe that $\UU$ and $\KK$ are also qualified to define the norm $\|\ \|_{C^m,\Omega',\UU,\KK}$ for any compact subdomain $\Omega'\subset\Omega$. It is well known that $\|\ \|_{C^m,\Omega,\UU,\KK}$ is equivalent to the norm $\|\ \|_{C^m,\Omega,g}$ defined by
	\[
		\|T\|_{C^m,\Omega,g}=\max_{0\le l\le m}\max_{x\in\Omega}|\nabla^lT(x)|\;;
	\]
i.e., there is some $C\ge1$, depending on $M$, $\Omega$, $\UU$, $\KK$, $g$ and $m$, such that
	\begin{equation}\label{norm equiv}
  		\frac{1}{C}\,\|\ \|_{C^m,\Omega,\UU,\KK}\le\|\ \|_{C^m,\Omega,g}
  		\le C\,\|\ \|_{C^m,\Omega,\UU,\KK}\;.
	\end{equation}
In particular, for $m=0$ and $f\in C^\infty(M)$,
	\begin{equation}\label{|f|_Omega}
		\|f\|_\Omega:=\|f\|_{C^0,\Omega,\UU,\KK}=\|f\|_{C^0,\Omega,g}=\max_{x\in\Omega}|f(x)|\;,
	\end{equation}
which is independent of the choices $\UU$, $\KK$ and $g$. 

The norms $\|\ \|_{C^m,\Omega,\UU,\KK}$ and $\|\ \|_{C^m,\Omega,g}$ have straightforward extensions to tensors with values in a separable Hilbert space $\E$, and satisfy the obvious versions of~\eqref{norm equiv} and~\eqref{|f|_Omega}, and $C^k(M,\E)$ is assumed to be equipped with the $C^k$ topology ($k\in\N\cup\{\infty\}$). 

For $f\in C^\infty(M,\E)$, recall that $\nabla f=df$ (its de~Rham differential). For each $m$, the map
	\[
		f_*^{(m)}\equiv\left(f_*^{(m),1},\dots,f_*^{(m),2^m}\right):T^{(m)}M\to T^{(m)}\E\equiv\E^{2^m}
	\]
is also $C^\infty$ and with values in a separable Hilbert space. In the following lemma, we consider the local representations of $f$ and every $f_*^{(m),\lambda}$ with respect to coordinate systems $(U,x^1,\dots,x^n)$ and $(U^{(m)},x_{(m)}^{1},\dots,x_{(m)}^{2^mn})$ of $M$ and $T^{(m)}M$. Moreover each function on $M$ or $U$ is identified with its lift to $T^{(m)}M$ or $U^{(m)}$.

\begin{lem}\label{l: f_*^(m)}
  	The following properties hold:
		\begin{enumerate}[{\rm(}i\/{\rm)}]
	
  			\item\label{i: f_*^(m)} The local representation of every $f_*^{(m),\lambda}$ is a universal polynomial expression of $x_{(m)}^{n+1},\dots,x_{(m)}^{2^mn}$ and the partial derivatives up to order $m$ of the local representation of $f$.
		
  			\item\label{i: partial up to order m of f} For each $\rho>0$, the partial derivatives up to order $m$ of the local representation of $f$ are given by universal linear expressions of the functions $(\sigma^{(m)}_{\rho,\mu})^*f_*^{(m),\lambda}$ for $n+1\le\mu\le2^mn$, where $\sigma^{(m)}_{\rho,\mu}:U\to U^{(m)}$ is the section of $\pi:U^{(m)}\to U$ determined by\footnote{Kronecker's delta is used here.} $(\sigma^{(m)}_{\rho,\mu})^*x_{(m)}^{\nu}=\rho\delta_{\mu\nu}$ for $n+1\le\nu\le2^mn$.
		
  		\end{enumerate}
\end{lem}

\begin{proof}
	By using induction on $m$, the result clearly boils down to the case $m=1$. But, in this case, the statement follows because $f_*\equiv(f,df):TM\to T\E\equiv\E^2$. 
\end{proof}

By using the supremum on $\Omega$ instead of the maximum, the definition of $\|\ \|_{C^m,\Omega,g}$ can be extended to any non-compact $n$-submanifold $\Omega\subset M$ (including $\Omega=M$), with possible infinite values. The tensors on $\Omega$ with finite norm $\|\ \|_{C^m,\Omega,g}$ are said to be \emph{uniformly $C^m$}, or $C^m_b$. For a given type, they form a Banach space, and the corresponding projective limit as $m\to\infty$ is a Fr\'echet space, whose elements are said to be \emph{uniformly $C^\infty$}, or $C^\infty_b$ (see e.g.\ \cite[Definition~2.7]{Roe1988I} or \cite[Definition~3.15]{Schick1996}). In particular, this gives rise to the Fr\'echet spaces $C^\infty_b(\Omega)$ and $C^\infty_b(\Omega,\E)$ when $\R$-valued and $\E$-valued $C^\infty_b$ functions are considered.

Let $N$ be another Riemannian manifold. Recall that a $C^1$ map $\phi\colon M \to N$ is called a ($\lambda$-) \emph{quasi-isometry}, or ($\lambda$-) \emph{quasi-isometric}, if there is some $\lambda \geq 1$ such that $\frac{1}{\lambda}\,|\xi|\le|\phi_*(\xi)| \leq \lambda \,|\xi|$ for every $\xi\in TM$; in particular, $\phi$ is an immersion. To define higher order quasi-isometries, let $T^{\le r}M=\{\,\xi\in TM\mid|\xi|\le r\,\}$ for each $r>0$. If $M$ has no boundary, then $T^{\le r}M$ is a manifold with boundary; otherwise, it is a manifold with corners. Also, define $T^{(m),\le r}M$ by induction on $m\in\Z^+$, setting $T^{(1),\le r}M=T^{\le r}M$ and $T^{(m),\le r}M=T^{\le r}T^{(m-1),\le r}M$.  It is said that $\phi:M\to N$ is a ($\lambda$-) \emph{quasi-isometry of order $m\in\N$}, or a ($\lambda$-) \emph{quasi-isometric map of order $m$}, if it is $C^{m+1}$ and $\phi_*^{(m)}:T^{(m),\le1}M\to T^{(m)}N$ is a ($\lambda$-) quasi-isometry. If $\phi$ is a quasi-isometry of order $m$ for all $m\in\N$, then it is called a \emph{quasi-isometry of order $\infty$}. If there is a quasi-isometric diffeomorphism $M\to N$ of order $m\in\N\cup\{\infty\}$, then $M$ and $N$ are said to be \emph{quasi-isometric with order $m$}. The property of being a quasi-isometry of order $m$ is preserved by the operations of composition of maps and inversion of diffeomorphisms \cite[Proposition~3.9]{AlvarezBarralCandel2016}, and therefore it induces an equivalence relation between Riemannian manifolds.

For $m\in\N$, a partial map $\phi:M\rightarrowtail N$ is called a \emph{$C^m$ local diffeomorphism}\footnote{The term ``$C^m$ local diffeomorfism'' ($m\ge1$) is also used in the standard sense, referring to any $C^m$ map $M\to N$ whose tangent map is an isomorphism at every point of $M$. The context will always clarify this ambiguity.} if $\dom\phi$ and $\im\phi$ are open in $M$ and $N$, respectively, and $\phi:\dom\phi\to\im\phi$ is a $C^m$ diffeomorphism. If moreover $\phi(x)=y$ for distinguished points, $x\in\dom\phi$ and $y\in\im\phi$, then it is said that $\phi:(M,x)\rightarrowtail(N,y)$ is a \emph{pointed $C^m$ local diffeomorphism}. For $m\in\N$, $R>0$ and $\lambda\ge1$, a $C^{m+1}$ pointed local diffeomorphism $\phi \colon (M,x) \rightarrowtail (N,y)$ is called an \emph{$(m,R,\lambda)$-pointed local quasi-isometry}, or a \emph{local quasi-isometry} of \emph{type} $(m,R,\lambda)$, if the restriction $\phi_*^{(m)}:\Omega^{(m)}\to T^{(m)}N$ is a $\lambda$-quasi-isometry for some compact domain $\Omega^{(m)}\subset\dom\phi_*^{(m)}$ with $B_M^{(m)}(x,R)\subset\Omega^{(m)}$ \cite[Definition~4.2]{AlvarezBarralCandel2016}.

\section{(Partial) quasi-equivalences}\label{s: (partial) q.-e.}

Let $M$ and $N$ be Riemannian $n$-manifolds, let $f\in C^\infty(M,\E)$ and $h\in C^\infty(N,\E)$, and let $x\in M$ and $y\in N$. Recall from Section~\ref{s: intro} the concepts of an equivalence $(M,f)\to(N,h)$, and a pointed equivalence $(M,f,x)\to(N,h,y)$. Observe that $\|f_*^{(m)}\|_{\Omega^{(m)}}$ makes sense for any $n$-submanifold $\Omega^{(m)}\subset T^{(m)}M$ because we consider $f_*^{(m)}:T^{(m)}M\to T^{(m)}\E\equiv\E^{2^m}$, with values in a separable Hilbert space. Note also that $(\phi^*h)_*^{(m)}=h_*^{(m)}\circ\phi_*^{(m)}$ for any $C^m$ map $\phi:M\to N$.

\begin{defn}\label{d: quasi-equivalence of order m}
	Let $\lambda\ge1$ and $\epsilon\ge0$, and let $\phi:M\to N$ be a $C^1$ map. It is said that $\phi:(M,f)\to(N,h)$ is a ($(\lambda,\epsilon)$-) \emph{quasi-equivalence of order $m\in\N$} if it is $C^{m+1}$, $\phi_*^{(m)}:T^{(m),\le1}M\to T^{(m)}N$ is a ($\lambda$-) quasi-isometry, and $\|f_*^{(m)}-(\phi^*h)_*^{(m)}\|_{T^{(m)}M}\le\epsilon$. If moreover distinguished points $x$ and $y$ are preserved, then $\phi:(M,f,x)\to(N,h,y)$ is called a \emph{pointed quasi-equivalence of order $m$}.  If there is a quasi-equivalence $(M,f)\to(N,h)$ (respectively, $(M,f,x)\to(N,h,y)$), then $(M,f)$ and $(N,h)$ (respectively, $(M,f,x)$ and $(N,h,y)$) are called \emph{quasi-equivalent}.
\end{defn}

\begin{rem}\label{r: q.-e. of order m}
  	\begin{enumerate}[(i)]
	
  		\item\label{i: q-e of order m are q-e of order m-1} Any $(\lambda,\epsilon)$-quasi-equivalence of order $m\ge1$ is a $(\lambda,\epsilon)$-quasi-equivalence of order $m-1$.
		
  		\item\label{i: phi_*^(m') is a lambda-q.-e. of order m-m'} For integers $0\le m'\le m$, if $\phi$ is a $(\lambda,\epsilon)$-quasi-equivalence of order $m$, then $\phi_*^{(m')}$ is a $(\lambda,\epsilon)$-quasi-equivalence of order $m-m'$.
		
  \end{enumerate}
\end{rem}

For a submanifold $\Omega\subset M$ and $f\in C^\infty(M,\E)$, the notation $(\Omega,f)$ is used for $(\Omega,f|_\Omega)$.

\begin{prop}\label{p: q.-e. of order m}
  	The following properties hold for any $m\in\N$, $\lambda,\mu\ge1$ and $\epsilon,\delta\ge0$:
  \begin{enumerate}[{\rm(}i\/{\rm)}]
	
  \item\label{i: composites of q.-e. of order m} There is some $\nu\ge1$, depending on $m$, $\lambda$ and $\mu$, such that, if $\phi:(M,f)\to(N,h)$ is a $(\lambda,\epsilon)$-quasi-equivalence and $\psi:(N,h)\to(L,u)$ a $(\mu,\delta)$-quasi-equivalence, both of them of order $m$, then $\psi\circ\phi:(M,f)\to(L,u)$ is a $(\nu,\epsilon+\delta)$-quasi-equivalence of order $m$.
		
  \item\label{i: inverses of q.-e. of order m} There are some $\nu'\ge1$, depending on $m$ and $\lambda$, such that, if $\phi:(M,f)\to(N,h)$ is a $(\lambda,\epsilon)$-quasi-equivalence of order $m$ and a diffeomorphism, then $\phi^{-1}:(N,h)\to(M,f)$ is a $(\nu',\epsilon)$-quasi-equivalence of order $m$.
		
  \end{enumerate}
\end{prop}

\begin{proof}
  	By \cite[Proposition~3.9]{AlvarezBarralCandel2016}, we only have to check the conditions on the $\E$-valued functions. Thus~\eqref{i: composites of q.-e. of order m} follows because, for each $\xi\in T^{(m)}M$, we have
		\begin{multline*}
			\left\|f_*^{(m)}(\xi)-((\psi\circ\phi)^*u)_*^{(m)}(\xi)\right\|\\
			\le\left\|f_*^{(m)}(\xi)-(\phi^*h)_*^{(m)}(\xi)\right\|
			+\left\|h_*^{(m)}\left(\phi_*^{(m)}(\xi)\right)-(\psi^*u)_*^{(m)}\left(\phi_*^{(m)}(\xi)\right)\right\|
			\le\epsilon+\delta\;.
		\end{multline*}
	Similarly,~\eqref{i: inverses of q.-e. of order m} follows because, for each $\zeta\in T^{(m)}N$,
		\[
			\left\|h_*^{(m)}(\zeta)-((\phi^{-1})^*f)_*^{(m)}(\zeta)\right\|
			=\left\|(\phi^*h)_*^{(m)}\left((\phi^{-1})_*^{(m)}(\zeta)\right)
			-f_*^{(m)}\left((\phi^{-1})_*^{(m)}(\zeta)\right)\right\|
			\le\epsilon\;.\qed
		\]
\renewcommand{\qed}{}
\end{proof}

\begin{cor}\label{c: being q.-e. is an equiv rel}
  ``Being quasi-equivalent with order $m$'' is an equivalence relation on the sets of pairs $(M,f)$ and triples $(M,f,x)$.
\end{cor}

Now, suppose that $M$ and $N$ are connected, complete and without boundary.

\begin{defn}\label{d: (m,R,lambda,epsilon)-pointed local quasi-equivalence}
  	Fix $m\in\N$, $R>0$, $\lambda\ge1$ and $\epsilon\ge0$. Let $\phi \colon (M,x) \rightarrowtail (N,y)$ be a $C^{m+1}$ pointed local diffeomorphism, and let $f\in C^\infty(M,\E)$ and $h\in C^\infty(N,\E)$. It is said that $\phi \colon (M,f,x) \rightarrowtail (N,h,y)$ is an \emph{$(m,R,\lambda,\epsilon)$-pointed local quasi-equivalence}, or a \emph{local quasi-equivalence} of \emph{type} $(m,R,\lambda,\epsilon)$, if there is some compact domain $\Omega^{(m)}\subset\dom\phi_*^{(m)}$ such that $B_M^{(m)}(x,R)\subset\Omega^{(m)}$ and $\phi_*^{(m)}:(\Omega^{(m)},f_*^{(m)})\to(T^{(m)}N,h_*^{(m)})$ is a $(\lambda,\epsilon)$-quasi-equivalence.
\end{defn}

\begin{rem}\label{r: (m,R,lambda,epsilon)-...}
  \begin{enumerate}[(i)]
	
  \item\label{i: ... (m,R,lambda,epsilon) is ... (m',R',lambda',epsilon')} Any pointed local quasi-equivalence $(M,f,x)\rightarrowtail(N,h,y)$ of type $(m,R,\lambda,\epsilon)$ is also of type $(m',R',\lambda',\epsilon')$ for $0\leq m'\leq m$, $0<R'<R$, $\lambda'>\lambda$ and $\epsilon'>\epsilon$.
		
  \item\label{i: ... (m,R,lambda,epsilon) if and only if ... (m-m',R,lambda,epsilon)} Consider integers $0\leq m' \leq m$, any pointed $C^{m+1}$ local diffeomorphism $\phi:(M,x)\rightarrowtail(N,y)$, and any $f\in C^\infty(M,\E)$ and $h\in C^\infty(N,\E)$. Then $\phi:(M,f,x)\rightarrowtail(N,h,y)$ is a pointed local quasi-equivalence of type $(m,R,\lambda,\epsilon)$ if and only if $\phi_*^{(m')}:(T^{(m')}M,f_*^{(m')},x)\rightarrowtail(T^{(m')}N,h_*^{(m')},y)$ is a pointed local quasi-equivalence of type $(m-m',R,\lambda,\epsilon)$.
		
  \item\label{i: C^infty approximation of (m,R,lambda,epsilon)-...} If there is an $(m,R,\lambda,\epsilon)$-pointed local quasi-equivalence $(M,f,x)\rightarrowtail(N,h,y)$, then, for all $R'<R$, $\lambda'>\lambda$ and $\epsilon'>\epsilon$, there is a $C^\infty$ $(m,R',\lambda',\epsilon')$-pointed local quasi-equivalence $(M,f,x)\rightarrowtail(N,h,y)$ by \cite[Theorem~2.7]{Hirsch1976}.
		
  \end{enumerate}
\end{rem}

\begin{lem}\label{l: composition and inversion of pointed local quasi-equivalences}
  The following properties hold:
  \begin{enumerate}[{\rm(}i\/{\rm)}]
		
  \item\label{i: composition} If $\phi: (M,f,x) \rightarrowtail (N,h,y)$ and $\psi: (N,h,y)\rightarrowtail(L,u,z)$ are pointed local quasi-equivalences of types $(m,R,\lambda,\epsilon)$ and $(m,\lambda R,\lambda',\epsilon')$, respectively, then $\psi\circ\phi:(M,f,x)\rightarrowtail (L,u,z)$ is an $(m,R,\lambda\lambda',\epsilon+\epsilon')$-pointed local quasi-equivalence.
			
  \item\label{i: inversion} If $\phi: (M,f,x) \rightarrowtail (N,h,y)$ is an $(m,\lambda R,\lambda,\epsilon)$-pointed local quasi-isometry, then $\phi^{-1}:(N,h,y)\rightarrowtail(M,f,x)$ is an $(m,R,\lambda,\epsilon)$-pointed local quasi-isometry.
			
  \end{enumerate}
\end{lem}

\begin{proof}
  	To prove~\eqref{i: composition}, take compact domains, $\Omega^{(m)}\subset T^{(m)}M$ and $\Omega^{\prime(m)}\subset T^{(m)}N$, such that $B_M^{(m)}(x,R)\subset\Omega^{(m)}$, $B_N^{(m)}(x,\lambda R)\subset\Omega^{\prime(m)}$, $\phi_*^{(m)}:(\Omega^{(m)},f_*^{(m)})\to(T^{(m)}N,h_*^{(m)})$ is a $(\lambda,\epsilon)$-quasi-equivalence, and $\psi_*^{(m)}:(\Omega^{\prime(m)},h_*^{(m)})\to(T^{(m)}L,u_*^{(m)})$ is a $(\lambda',\epsilon')$-quasi-equivalence. According to the proof of \cite[Lemma~4.3-(i)]{AlvarezBarralCandel2016}, there is a compact domain $\Omega_0^{(m)}\subset T^{(m)}M$ such that $B_M^{(m)}(x,R)\subset\Omega_0^{(m)}$ and $\phi_*^{(m)}(\Omega_0^{(m)})\subset\Omega^{\prime(m)}$. Then $(\psi\circ\phi)_*^{(m)}:\Omega_0^{(m)}\to T^{(m)}L$ is a $\lambda\lambda'$-quasi-isometry by \cite[Remark~2-(v)]{AlvarezBarralCandel2016}. Moreover, for each $\xi\in\Omega_0^{(m)}$,
		\begin{multline*}
			\left\|f_*^{(m)}(\xi)-((\psi\circ\phi)^*u)_*^{(m)}(\xi)\right\|\\
			\le\left\|f_*^{(m)}(\xi)-(\phi^*h)_*^{(m)}(\xi)\right\|
			+\left\|h_*^{(m)}\left(\phi_*^{(m)}(\xi)\right)-(\psi^*u)_*^{(m)}\left(\phi_*^{(m)}(\xi)\right)\right\|
			\le\epsilon+\epsilon'\;.
		\end{multline*}
	So $\psi\circ\phi:(M,f,x)\rightarrowtail (L,u,z)$ is an $(m,R,\lambda\lambda',\epsilon+\epsilon')$-pointed local quasi-equivalence.
	
	To prove~\eqref{i: inversion}, let $\Omega^{(m)}\subset T^{(m)}M$ be a compact domain such that $B_M^{(m)}(x,R)\subset\Omega^{(m)}$, and $\phi_*^{(m)}:(\Omega^{(m)},f_*^{(m)})\to(T^{(m)}N,h_*^{(m)})$ is a $(\lambda,\epsilon)$-quasi-equivalence. According to the proof of \cite[Lemma~4.3-(ii)]{AlvarezBarralCandel2016}, the compact domain $\Omega^{\prime(m)}:=\phi_*^{(m)}(\Omega^{(m)})\subset T^{(m)}N$ contains $B_N^{(m)}(y,R)$. Then $(\phi^{-1})_*^{(m)}=(\phi_*^{(m)})^{-1}:\Omega^{\prime(m)}\to T^{(m)}M$ is a $\lambda$-quasi-isometry by \cite[Remark~2-(vi)]{AlvarezBarralCandel2016}. Moreover, for each $\xi\in\Omega^{\prime(m)}$,
		\[
			\left\|h_*^{(m)}(\xi)-((\phi^{-1})^*f)_*^{(m)}(\xi)\right\|
			\le\left\|(\phi^*h)_*^{(m)}\left((\phi^{-1})_*^{(m)}(\xi)\right)-f_*^{(m)}\left((\phi^{-1})_*^{(m)}(\xi)\right)\right\|
			\le\epsilon\;.
		\]
	So $\phi^{-1}:(N,h,y)\rightarrowtail (M,f,x)$ is an $(m,R,\lambda,\epsilon)$-pointed local quasi-equivalence.
\end{proof}

\section{The $C^\infty$ topology on $\widehat\MM_*(n)$}\label{s: C^infty topology}

\begin{defn}\label{d: widehat U^m_R,r}
  For $m\in\N$ and $R,r>0$, let $\widehat U^m_{R,r}$ be the set of pairs
  $([M,f,x],[N,h,y])\in\widehat\MM_*(n)\times\widehat\MM_*(n)$ such that there is some
  $(m,R,\lambda,\epsilon)$-pointed local quasi-equivalence $(M,f,x)\rightarrowtail
  (N,h,y)$ for some $\lambda\in[1,e^r)$ and $\epsilon\in(0,r)$.
\end{defn}

\begin{prop}\label{p: C^infty uniformity in widehat MM_*(n)}
  	The following properties\footnote{The following standard notation is used for a set $X$ and relations $U,V\subset X\times X$:
		\begin{align*}
			U^{-1}&=\{\,(y,x)\in X\times X\mid(x,y)\in U\,\}\;,\\
			V\circ U&=\{\,(x,z)\in X\times X\mid\exists y\in X\ \text{so that}\ (x,y)\in U\ \text{and}\ (y,z)\in V\,\}\;.
		\end{align*}
	Moreover the diagonal of $X\times X$ is denoted by $\Delta$.} hold for all $m,m'\in\N$ and $R,S,r,s>0$:
		\begin{enumerate}[{\rm(}i\/{\rm)}]
  	
			\item\label{i: ^-1} $(\widehat U^m_{e^rR,r})^{-1} \subset\widehat U^m_{R,r}$.
		
			\item\label{i: cap} $\widehat U^{m_0}_{R_0,r_0}\subset\widehat U^m_{R,r}\cap\widehat U^{m'}_{S,s}$, where $m_0=\max\{m,m'\}$, $R_0=\max\{R,S\}$ and $r_0=\min\{r,s\}$.
		
			\item\label{i: Delta} $\Delta\subset\widehat U^m_{R,r}$.
		
			\item\label{i: circ} $\widehat U^m_{R,r}\circ\widehat U^m_{e^rR,s} \subset\widehat U^m_{R,r+s}$.
		
		\end{enumerate}
\end{prop}

\begin{proof}
  	Properties~\eqref{i: cap} and~\eqref{i: Delta} are elementary, and~\eqref{i: ^-1} and~\eqref{i: circ} are consequences of Lemma~\ref{l: composition and inversion of pointed local quasi-equivalences}.
\end{proof}

\begin{prop}\label{p: Hausdorff widehat MM_*(n)}
$\bigcap_{R,r>0}\widehat U^m_{R,r}=\Delta$ for all $m\in\N$.
\end{prop}

\begin{proof}
  	We only have to prove ``$\subset$'' by Proposition~\ref{p: C^infty uniformity in widehat MM_*(n)}-\eqref{i: Delta}. For $([M,f,x],[N,h,y])\in\bigcap_{R,r>0}\widehat U^m_{R,r}$, there is a sequence of pointed local quasi-equivalences $\phi_i:(M,f,x)\rightarrowtail(N,h,y)$, with corresponding types $(m,R_i,\lambda_i,\epsilon_i)$, such that $R_i\uparrow\infty$, $\lambda_i\downarrow 1$ and $\epsilon_i\downarrow 0$ as $i\to\infty$. According to the proof of  \cite[Proposition~5.3]{AlvarezBarralCandel2016}, for each $i$, there is some subsequence $\phi_{k(i,l)}$ whose restriction to $B_M(x,R_i)$ converges to some pointed isometric immersion $\psi_i:(B_M(x,R_i),x)\to(N,y)$ in the weak $C^m$ topology, $\psi_{i+1}|_{B_M(x,R_i)}=\psi_i$ for all $i$, and the combination of the maps $\psi_i$ is a pointed isometry $\psi:(M,x)\to(N,y)$. For every $x'\in M$ and $\epsilon>0$, there are some $i$ and $\delta>0$ so that $x'\in B_M(x,R_i)$, $\epsilon_i\le\epsilon/2$, and $\|h(y')-h(y'')\|<\epsilon/2$ if $d_N(y',y'')<\delta$ for all $y',y''\in\ol B_M(x,R_i)$. Moreover there is some $l$ such that $d_N(\phi_{k(i,l)}(x'),\psi_i(x'))<\delta$. Hence
		\[
			\|f(x')-h\circ\psi(x')\|
			\le\|f(x')-h\circ\phi_{k(i,l)}(x')\|+\|h\circ\phi_{k(i,l)}(x')-h\circ\psi(x')\|
			<\epsilon_i+\epsilon/2\le\epsilon.
		\]
	Since $x'$ and $\epsilon$ are arbitrary, it follows that $\psi:(M,f,x)\to(N,h,y)$ is an equivalence, and therefore $[M,f,x]=[N,h,y]$. 
\end{proof}

By Propositions~\ref{p: C^infty uniformity in widehat MM_*(n)} and~\ref{p: Hausdorff widehat MM_*(n)}, the sets $\widehat U_{R,r}^m$ form a base of entourages of a separating uniformity on $\widehat\MM_*(n)$, which is called the \emph{$C^\infty$ uniformity}. 
 
\begin{defn}\label{d: widehat D^m_R,r}
  	For $R,r>0$ and $m\in\N$, let $\widehat D^m_{R,r}$ be the set of pairs $([M,f,x],[N,h,y])\in\widehat\MM_*(n)\times\widehat\MM_*(n)$ such that there is some $C^{m+1}$ pointed local diffeomorphism $\phi\colon (M,x) \rightarrowtail (N,y)$ so that $\|g_M-\phi^*g_N\|_{C^m,\Omega,g_M}<r$ and $\|f-\phi^*h\|_{C^m,\Omega,g_M}<r$ for some compact domain $\Omega\subset\dom\phi$ with $B_M(x,R)\subset\Omega$.
\end{defn} 
 
 \begin{rem}\label{r: widehat D^m_R,r}
   	By~\eqref{norm equiv}, and its version for $\E$-valued functions, a sequence $[M_i,f_i,x_i]\in\widehat\MM_*(n)$ is $C^\infty$ convergent to $[M,f,x]\in\widehat\MM_*(n)$ if and only if it is eventually in\footnote{Given a set $X$, for $U\subset X\times X$ and $x\in X$, let $U(x)=\{\,y\in Y\mid(x,y)\in U\,\}$. In the case of $U\subset\widehat\MM_*(n)\times\widehat\MM_*(n)$ and $[M,f,x]\in\widehat\MM_*(n)$, we simply write $U(M,f,x)$.} $\widehat D^m_{R,r}(M,f,x)$ for arbitrary $m\in\N$ and $R,r>0$.
 \end{rem}
 
\begin{prop}\label{p: widehat D^m_R,r'(M,f,x) subset widehat U^m_R,r(M,f,x)}
  	The following properties hold:
  		\begin{enumerate}[{\rm(}i\/{\rm)}]

  			\item\label{i: widehat D^0_R,r' subset widehat U^0_R,r} For all $R,r>0$, if $0<r'\le\min\{1-e^{-2r},e^{2r}-1,r\}$, then $\widehat D^0_{R,r'}\subset\widehat U^0_{R,r}$.
		
  			\item\label{i: widehat D^m_R,r'(M,f,x) subset widehat U^m_R,r(M,f,x)} For all $m\in\Z^+$, $R,r>0$ and $[M,f,x]\in\widehat\MM_*(n)$, there is some $r'>0$ such that $\widehat D^m_{R,r'}(M,f,x)\subset\widehat U^m_{R,r}(M,f,x)$.

  		\end{enumerate}
\end{prop}

\begin{proof}
  	Let us show~\eqref{i: widehat D^0_R,r' subset widehat U^0_R,r}. If $([M,f,x],[N,h,y])\in\widehat D^0_{R,r'}$, then there is a $C^1$ pointed local diffeomorphism $\phi:(M,x)\rightarrowtail(N,y)$ such that $r'_0:=\|g_M-\phi^*g_N\|_{C^0,\Omega,g_M}<r'$ and $\epsilon_0:=\|f-\phi^*h\|_{C^0,\Omega,g_M}<r'$ for some compact domain $\Omega\subset\dom\phi$ with $B_M(x,R)\subset\Omega$. Take some $\lambda\in[1,e^r)$ such that $r'_0\le\min\{1-\lambda^{-2},\lambda^2-1\}$. According to the proof of \cite[Proposition~6.4-(i)]{AlvarezBarralCandel2016}, $\phi:\Omega\to N$ is a $\lambda$-quasi-isometry. Since moreover $\|f-\phi^*h\|_{\Omega}\le\epsilon_0$, it follows that $\phi$ is a $(0,R,\lambda,r'_0,\epsilon_0)$-pointed local quasi-equivalence, obtaining that $([M,f,x],[N,h,y])\in\widehat U^0_{R,r}$.

  	Let us prove~\eqref{i: widehat D^m_R,r'(M,f,x) subset widehat U^m_R,r(M,f,x)}. Take $m\in\Z^+$, $R, r>0$ and $[M,f,x]\in\widehat\MM_*(n)$. Let $\UU$ be a finite collection of charts of $M$ with domains $U_a$, and let $\KK=\{K_a\}$ be a family of compact subsets of $M$, with the same index set as $\UU$, such that $K_a\subset U_a$ for all $a$, and $\ol B_M(x,R)\subset\Int(K)$ for $K=\bigcup_aK_a$. Let $r'>0$, to be fixed later. For any $[N,h,y]\in\widehat D^m_{R,r'}(M,x)$, there is a $C^{m+1}$ pointed local diffeomorphism $\phi\colon(M,x)\rightarrowtail(N,y)$ so that $\|g_M-\phi^*g_N\|_{C^m,\Omega,g_M}<r'$ and $\|f-\phi^*h\|_{C^m,\Omega,g_M}<r'$ for some compact domain $\Omega\subset\dom\phi\cap\Int(K)$ with $B_M(x,R)\subset\Omega$. By continuity, there is another compact domain $\Omega'\subset\dom\phi\cap\Int(K)$ such that $\Omega\subset\Int(\Omega')$, $\|g_M-\phi^*g_N\|_{C^m,\Omega',g_M}<r'$ and $\|f-\phi^*h\|_{C^m,\Omega',g_M}<r'$. According to the proof of \cite[Proposition~6.4-(i)]{AlvarezBarralCandel2016}, if $r'$ is small enough (depending on $m$, $R$, $r$ and $[M,x]$), then there is some compact domain $\Omega^{(m)}\subset T^{(m)}M$ such that $B^{(m)}_M(x,R)\subset\Omega^{(m)}\subset\pi^{-1}(\Omega')$, where $\pi:T^{(m)}M\to M$, and $\phi_*^{(m)}:\Omega^{(m)}\to T^{(m)}N$ is a $\lambda$-quasi-isometry for some $\lambda\in[1,e^r)$. Given $\epsilon\in(0,r)$, choose some $C\ge1$ satisfying~\eqref{norm equiv} for $\E$-valued functions with $\UU$, $\KK$, $\Omega'$ and $g$, and, according to Lemma~\ref{l: f_*^(m)}-\eqref{i: f_*^(m)}, choose some $\epsilon'>0$ such that
    		\[
    			\|f-\phi^*h\|_{C^m,\Omega',\UU,\KK}<\epsilon'\,\Longrightarrow\,
    			\|f_*^{(m)}-(\phi^*h)_*^{(m)}\|_{\Omega^{(m)}}<\epsilon\;.
    		\]
  	Suppose that $r'\le\epsilon'/C$. Then
  		\[
    			\|f-\phi^*h\|_{C^m,\Omega',g_M}<r'
    			\,\Longrightarrow\,\|f-\phi^*h\|_{C^m,\Omega',\UU,\KK}<Cr'\le\epsilon'
    			\,\Longrightarrow\,\|f_*^{(m)}-(\phi^*h)_*^{(m)}\|_{\Omega^{(m)}}<\epsilon\;.
  		\]
  	Hence $\phi$ is an $(m,R,\lambda,\epsilon)$-pointed local quasi-equivalence $(M,f,x)\rightarrowtail (N,h,y)$, and therefore $[N,h,y]\in\widehat U^{(m)}_{R,r}(M,f,x)$.
\end{proof}

\begin{prop}\label{p: widehat U^m_R,r'(M,f,x) subset widehat D^m_R,r(M,f,x)}
  	The following properties hold:
		\begin{enumerate}[{\rm(}i\/{\rm)}]

  			\item\label{i: widehat U^0_R,r' subset widehat D^0_R,r} For all $R,r>0$, if $e^{2r'}-e^{-2r'}\le r$, then $\widehat U^0_{R,r'}\subset\widehat D^0_{R,r}$.
		
  			\item\label{i: widehat U^m_R,r'(M,f,x) subset widehat D^m_R,r(M,f,x)} For all $m\in\Z^+$, $R, r>0$ and $[M,f,x]\in\widehat\MM_*(n)$, there is some $r'>0$ such that $\widehat U^m_{R,r'}(M,f,x)\subset\widehat D^m_{R,r}(M,f,x)$.

  		\end{enumerate}
\end{prop}

\begin{proof}
  	Let us show~\eqref{i: widehat U^0_R,r' subset widehat D^0_R,r}. If $([M,f,x],[N,h,y])\in\widehat U^0_{R,r'}$, then there is a $(0,R,\lambda,\epsilon)$-pointed local quasi-equivalence $\phi:(M,f,x)\rightarrowtail(N,h,y)$ for some $\lambda\in[1,e^{r'})$ and $\epsilon\in(0,r')$. Thus there is some compact domain $\Omega\subset\dom\phi$ such that $B_M(x,R)\subset\Omega$ and $\phi:(\Omega,f)\to(N,h)$ is a $(\lambda,\epsilon)$-quasi-equivalence. According to the proof of \cite[Proposition~6.5-(i)]{AlvarezBarralCandel2016}, $\|g_M-\phi^*g_N\|_{C^0,\Omega,g}<r$. So $([M,f,x],[N,h,y])\in\widehat D^0_{R,r}$.

  	Let us prove~\eqref{i: widehat U^m_R,r'(M,f,x) subset widehat D^m_R,r(M,f,x)}. Let $m\in\Z^+$, $R, r>0$ and $[M,f,x]\in\widehat\MM_*(n)$. Take $\UU$, $\KK$ and $K$ like in the proof of Proposition~\ref{p: widehat D^m_R,r'(M,f,x) subset widehat U^m_R,r(M,f,x)}-\eqref{i: widehat D^m_R,r'(M,f,x) subset widehat U^m_R,r(M,f,x)}. Let $r'>0$, to be fixed later. For any $[N,h,y]\in\widehat U^m_{R,r'}(M,x)$, there is an $(m,R,\lambda,\epsilon)$-pointed local quasi-equivalence $\phi:(M,f,x)\rightarrowtail(N,h,y)$ for some $\lambda\in[1,e^{r'})$ and $\epsilon\in(0,r')$. Thus there is a compact domain $\Omega^{(m)}\subset\dom\phi_*^{(m)}\cap\Int(K^{(m)})$ so that $B_M^{(m)}(x,R)\subset\Omega^{(m)}$ and $\phi_*^{(m)}:(\Omega^{(m)},f_*^{(m)})\to(T^{(m)}N,h_*^{(m)})$ is a $(\lambda,\epsilon)$-quasi-equivalence. According to the proof of  \cite[Proposition~6.5-(ii)]{AlvarezBarralCandel2016}, there are compact domains, $\Omega'^{(m)}\subset\dom\phi_*^{(m)}$ and $\Omega\subset M$, such that $\Omega^{(m)}\subset\Int(\Omega'^{(m)})$, $\Omega^{(m)}\cap M\subset\Omega\subset\Int(\Omega'^{(m)})$, and $\|g_M-\phi^*g_N\|_{C^m,\Omega,g}<r$ if $r'$ is small enough; in particular, $B_M(x,R)\subset\Omega$ because $M$ is a totally geodesic Riemannian submanifold of $T^{(m)}M$. Take some $C\ge1$ satisfying~\eqref{norm equiv} for $\E$-valued functions with $\UU$, $\KK$, $\Omega$ and $g_M$. With the notation of Section~\ref{s: prelim}, for $\rho>0$ and $n+1\le\mu\le2^mn$, let $\sigma^{(m)}_{a,\rho,\mu}:U_a\to U_a^{(m)}$ be the section of each $\pi:U_a^{(m)}\to U_a$ of the type used in Lemma~\ref{l: f_*^(m)}-\eqref{i: partial up to order m of f}. Since $\Omega\subset\Int(\Omega'^{(m)})$, there is some $\rho>0$ so that $\sigma^{(m)}_{\rho,\mu}(K_a\cap\Omega)\subset\Omega'^{(m)}$ for all $a$ and $\mu$. Thus, by Lemma~\ref{l: f_*^(m)}-\eqref{i: partial up to order m of f}, there is some $\epsilon'>0$, depending on $r$ and $\rho$, such that
  		\[
  			\|f_*^{(m)}-(\phi^*h)_*^{(m)}\|_{\Omega'^{(m)}}<\epsilon'
			\,\Longrightarrow\,\|f^*-\phi^*h\|_{C^m,\Omega,\UU,\KK}<r/C\;.
  		\]
  	Suppose that moreover $r'<\epsilon'$, and therefore $\epsilon<\epsilon'$. Then
  		\[
    			\|f_*^{(m)}-(\phi^*h)_*^{(m)}\|_{\Omega'^{(m)}}\le\epsilon<\epsilon'
    			\,\Longrightarrow\,\|f-\phi^*h\|_{C^m,\Omega,\UU,\KK}<r/C
    			\,\Longrightarrow\,\|f-\phi^*h\|_{C^m,\Omega,g}<r\;,
  		\]
  	showing that $[N,h,y]\in\widehat D^{(m)}_{R,r}(M,f,x)$.
\end{proof}

\begin{cor}\label{c: C^infty convergence in MM_*(n)}
  	The $C^\infty$ convergence in $\widehat\MM_*(n)$ describes the topology induced by the $C^\infty$ uniformity.
\end{cor}

\begin{proof}
  	This is a direct consequence of Remark~\ref{r: widehat D^m_R,r} and Propositions~\ref{p: widehat D^m_R,r'(M,f,x) subset widehat U^m_R,r(M,f,x)} and~\ref{p: widehat U^m_R,r'(M,f,x) subset widehat D^m_R,r(M,f,x)}.
\end{proof}

According to Corollary~\ref{c: C^infty convergence in MM_*(n)}, the $C^\infty$ uniformity induces what was called the $C^\infty$ topology in Section~\ref{s: intro}. Recall that the corresponding space is denoted by $\widehat\MM_*^\infty(n)$, and the notation $\widehat{\Cl}_\infty$ is used for the closure operator in $\widehat\MM_*^\infty(n)$.

\begin{prop}\label{p: widehat MM_*^infty(n) is separable}
	$\widehat\MM_*^\infty(n)$ is separable.
\end{prop}

\begin{proof}
  	According to the proof of \cite[Proposition~7.1]{AlvarezBarralCandel2016}, there is a countable family $\mathcal{C}$ of $C^\infty$ compact manifolds containing exactly one representative of every diffeomorphism class, and, for every $M\in\CC$, there is a countable dense subset $\GG_M$ of the space of metrics on $M$ with the $C^\infty$ topology. Take also countable dense subsets, $\DD_M\subset M$ and $\FF_M\subset C^\infty(M,\E)$. Then, like in the proof of \cite[Proposition~7.1]{AlvarezBarralCandel2016}, the countable set
  		\begin{equation}\label{countable dense subset}
  			\{\,[(M,g),f,x]\mid M\in\mathcal{C},\ g\in\GG_M,\ x\in\DD_M,\ f\in\FF_M\,\}
  		\end{equation}
  	is dense in $\widehat\MM_*^\infty(n)$.
\end{proof}

\begin{prop}\label{p: widehat MM_*^infty(n) is completely metrizable}
	The $C^\infty$ uniformity is complete and metrizable.
\end{prop}

\begin{proof}
  	According to  \cite[Corollary~38.4]{Willard1970}, the $C^\infty$ uniformity on $\widehat\MM_*(n)$ is metrizable because it is separating and the sets $\widehat U_{k,1/k}$ ($k\in\Z^+$) form a countable base of entourages. To check that this uniformity is complete, consider an arbitrary Cauchy sequence $[M_i,f_i,x_i]$ in $\widehat\MM_*(n)$. We have to prove that $[M_i,f_i,x_i]$ is convergent in $\widehat\MM_*^\infty(n)$. By taking a subsequence if necessary, we can suppose that $([M_i,f_i,x_i],[M_{i+1},x_{i+1},f_{i+1}])\in U^{m_i}_{R_i,r_i}$ for sequences, $m_i\uparrow\infty$ in $\N$, and $R_i\uparrow\infty$ and $r_i\downarrow0$ in $\R^+$, such that $\sum_ir_i<\infty$, and $R_{i+1}\ge e^{r_i}R_i$ for all $i$. Let $\bar r_i=\sum_{j\ge i}r_j$. Consider other sequences $R'_i,R''_i\uparrow\infty$ in $\R^+$ such that $R'_i<R''_i\le e^{-\bar r_i}R_i$ and $R'_{i+1}\ge e^{r_i}R''_i$.

  	For each $i$, there is some $(m_i,R_i,\lambda_i,\epsilon_i)$-pointed local quasi-equivalence $\phi_i\colon(M_i,x_i)\rightarrowtail(M_{i+1},x_{i+1})$, for some $\lambda_i\in(1,e^{r_i})$ and $\epsilon_i\in(0,r_i)$. Then $\bar\lambda_i:=\prod_{j\ge i}\lambda_j<e^{\bar r_i}$ and $\bar\epsilon_i:=\sum_{j\ge i}\epsilon_j<\bar r_i$. Moreover each $\phi_i$ can be assumed to be $C^\infty$ by Remark~\ref{r: (m,R,lambda,epsilon)-...}-\eqref{i: C^infty approximation of (m,R,lambda,epsilon)-...}. For $i<j$, the pointed local quasi-equivalence $\psi_{ij}=\phi_{j-1}\circ\dots\circ\phi_i:(M_i,f_i,x_i)\rightarrowtail(M_j,x_j,f_j)$ is of type $(m_i,R_i/\bar\lambda_i,\bar\lambda_i,\bar r_i)$ by Lemma~\ref{l:  composition and inversion of pointed local quasi-equivalences}-\eqref{i: composition}.
	
	For $i,m\in\N$, let
  		\begin{alignat*}{3}
    			B_i&=B_i(x_i,R_i)\;,&\quad B_i'&=B_i(x_i,R'_i)\;,&\quad B_i''&=B_i(x_i,R''_i)\;,\\
    			B_i^{(m)}&=B_i^{(m)}(x_i,R_i)\;,&\quad
    			B_i'^{(m)}&=B_i^{(m)}(x_i,R'_i)\;,&\quad
    			B_i''^{(m)}&=B_i^{(m_i)}(x_i,R''_i)\;.
  		\end{alignat*}
  	A bar is added to this notation when the corresponding closed balls are considered. We have $\phi_i(\ol B_i)\subset B_{i+1}$ because $R_{i+1}>\lambda_iR_i$, and $\phi_{i*}^{(m_i)}(\ol B_i''^{(m_i)})\subset B_{i+1}'^{(m_i)}\subset B_{i+1}'^{(m_{i+1})}$ since $R'_{i+1}>\lambda_iR''_i$ and $g_{i+1}^{(m_i)}$ is the restriction of $g_{i+1}^{(m_i+1)}$. Furthermore $B_i''\subset\dom\psi_{ij}$ and $B_i''^{(m_i)}\subset\dom\psi_{ij*}^{(m_i)}$ for $i<j$ because $R''\le R_i/\bar\lambda_i$. Therefore $\psi_{ij}(B_i)\subset B_j$ and $\psi_{ij*}^{(m_i)}(B_i''^{(m_i)})\subset B_j'^{(m_j)}$. Take compact domains, $\Omega_i\subset M_i$ and $\Omega_i^{(m_i)}\subset T^{(m_i)}M_i$, such that $B_i'\subset\Omega_i\subset\Int(\Omega_i^{(m_i)})$ and $B_i'^{(m_i)}\subset\Omega_i^{(m_i)}\subset B_i''^{(m_i)}$; thus $\Omega_i\subset B_i''$ since $M_i$ is a totally geodesic Riemannian submanifold of $T^{(m_i)}M_i$.
  
  	According to the proof of \cite[Proposition~7.2]{AlvarezBarralCandel2016}, there is a pointed complete connected Riemannian manifold $(\widehat M,\hat x)$, and, for each $i$, there is some $C^\infty$ pointed map $\psi_i:(B_i,x_i)\to(\widehat M,\hat x)$ such that $\psi_{i*}^{(m_i)}:\Omega_i^{(m_i)}\to T^{(m_i)}\widehat M$ is a $\bar\lambda_i$-quasi-isometry, and $\psi_i=\psi_j\circ\psi_{ij}$ for $j\ge i$. Let $\widehat B_i=\psi_i(B_i)$, $\widehat\Omega_i=\psi_i(\Omega_i)$ and $\widehat\Omega_i^{(m_i)}=\psi_{i*}^{(m_i)}(\Omega_i^{(m_i)})$. Let $\hat f_i\in C^\infty(\widehat B_i,\E)$ be determined by $\psi_i^*\hat f_i=f_i|_{B_i}$. 
	
	\begin{claim}\label{cl: hat f_j}
    		For all $i$, the sequence $\hat f_j|_{\widehat\Omega_i}$ ($j\ge i$) is convergent in $C^{m_i}(\widehat\Omega_i,\E)$.
  	\end{claim}
	
  	This assertion follows by showing that the restrictions of the functions $f_{ij}:=\psi_{ij}^*f_j$ to $\Omega_i$, for $j\ge i$, form a convergent sequence in $C^{m_i}(\Omega_i,\E)$. Equivalently, we show that $f_{ij}|_{\Omega_i}$ is a Cauchy sequence with respect to $\|\ \|_{C^{m_i},\Omega_i,g_i}$. For $k\ge j\ge i$,
  		\begin{multline}\label{|f_ij*^(m_i) - f_ik*^(m_i)|_Omega_i^(m_i)}
  			\left\|f_{ij*}^{(m_i)}-f_{ik*}^{(m_i)}\right\|_{\Omega_i^{(m_i)}}
			=\left\|f_{j*}^{(m_i)}-f_{jk*}^{(m_i)}\right\|_{\psi_{ij}(\Omega_i^{(m_i)})}\\
  			\le\left\|f_{j*}^{(m_i)}-f_{j\,j+1*}^{(m_i)}\right\|_{\psi_{ij}(\Omega_i^{(m_i)})}+\dots
			+\left\|f_{k-1*}^{(m_i)}-f_{k-1\,k*}^{(m_i)}\right\|_{\psi_{i\,k-1*}^{(m_i)}(\Omega_i^{(m_i)})}\\
			\le\left\|f_{j*}^{(m_j)}-f_{j\,j+1*}^{(m_j)}\right\|_{\Omega_j^{(m_j)}}+\dots
			+\left\|f_{k-1*}^{(m_{k-1})}-f_{k-1\,k*}^{(m_{k-1})}\right\|_{\Omega_{k-1}^{(m_{k-1})}}
			\le\epsilon_j+\dots\epsilon_{k-1}<\bar\epsilon_j
  		\end{multline}
  	because
  		\[
  			\psi_{ij*}^{(m_i)}(\Omega_i^{(m_i)})\subset\psi_{ij*}^{(m_i)}(B_i''^{(m_i)})
  			\subset B_j'^{(m_j)}\subset\Omega_j^{(m_j)}
  		\]
  	and $f_{jk*}^{(m_j)}=f_{jk*}^{(m_i)}$ on $\Omega_j^{(m_j)}\cap B_j^{(m_i)}\supset\psi_{ij*}^{(m_i)}(\Omega_i^{(m_i)})$. 
	
  	Let $\UU_i$ be a finite collection of charts of $M_i$ with domains $U_{i,a}$, and let $\KK_i=\{K_{i,a}\}$ be a family of compact subsets of $M_i$, with the same index set as $\UU_i$, such that $K_{i,a}\subset U_{i,a}$ for all $a$, and $\ol B_i''\subset\bigcup_aK_{i,a}=:K_i$. Thus $\Omega_i\subset K_i$. Choose some $C_i\ge1$ satisfying~\eqref{norm equiv} for $\E$-valued functions with $\UU_i$, $\KK_i$, $\Omega_i$ and $g_i$. With the notation of Section~\ref{s: prelim}, for any $\rho>0$ and $n+1\le\mu\le2^{m_i}n$, let $\sigma^{(m_i)}_{i,a,\rho,\mu}:U_{i,a}\to U_{i,a}^{(m_i)}$ be the section of each $\pi:U_{i,a}^{(m_i)}\to U_{i,a}$ of the type used in Lemma~\ref{l: f_*^(m)}-\eqref{i: partial up to order m of f}. Since $\Omega_i\subset\Int(\Omega_i^{(m_i)})$, there is some $\rho>0$ so that $\sigma^{(m_i)}_{i,a,\rho,\mu}(K_{i,a}\cap\Omega_i)\subset K_{i,a}^{(m_i)}\cap\Omega_i^{(m_i)}$ for all $a$ and $\mu$. Thus, by Lemma~\ref{l: f_*^(m)}-\eqref{i: partial up to order m of f}, given any $\epsilon>0$, there is some $\delta>0$, depending on $\epsilon$ and $\rho$, such that
  		\begin{equation}\label{... < delta Longrightarrow ... < epsilon/C_i}
    			\|f_{ij*}^{(m_i)}-f_{ik*}^{(m_i)}\|_{\Omega_i^{(m_i)}}<\delta
    			\,\Longrightarrow\,\|f_{ij}-f_{ik}\|_{C^m,\Omega_i,\UU_i,\KK_i}<\epsilon/C_i\;.
  		\end{equation}
  	For $j$ large enough, we have $\bar\epsilon_j<\delta$, giving
  		\[
    			\|f_{ij*}^{(m_i)}-f_{ik*}^{(m_i)}\|_{\Omega_i^{(m_i)}}<\delta\\
    			\,\Longrightarrow\,\|f_{ij}-f_{ik}\|_{C^{m_i},\Omega_i,\UU_i,\KK_i}<\epsilon/C_i
    			\,\Longrightarrow\,\|f_{ij}-f_{ik}\|_{C^{m_i},\Omega_i,g_i}<\epsilon
 		 \]
  	by~\eqref{|f_ij*^(m_i) - f_ik*^(m_i)|_Omega_i^(m_i)},~\eqref{... < delta Longrightarrow ... < epsilon/C_i} and~\eqref{norm equiv}. This shows that $f_{ij}|_{\Omega_i}$ is a Cauchy sequence in the Banach space $C^{m_i}(\Omega_i,\E)$ with $\|\ \|_{C^{m_i},\Omega_i,g_i}$, and therefore it is convergent. This completes the proof of Claim~\ref{cl: hat f_j}.
	
  	According to Claim~\ref{cl: hat f_j}, for each $i$, let $\hat f_{i\infty}=\lim_{k\to\infty}\hat f_k|_{\widehat\Omega_i}$ in $C^{m_i}(\widehat\Omega_i,\E)$. Obviously, $\hat f_{j\infty}|_{\widehat\Omega_i}=\hat f_{i\infty}$ for $j>i$. Hence there is a function $\hat f\in C^\infty(\widehat M,\E)$ whose restriction to every $\widehat\Omega_i$ is $\hat f_{i\infty}$. From~\eqref{|f_ij*^(m_i) - f_ik*^(m_i)|_Omega_i^(m_i)}, we get $\|\hat f_{i*}^{(m_i)}-\hat f_{k*}^{(m_i)}\|_{\widehat\Omega_i^{(m_i)}}<\bar\epsilon_i$ for $k\ge i$, yielding $\|\hat f_{i*}^{(m_i)}-\hat f_*^{(m_i)}\|_{\widehat\Omega_i^{(m_i)}}\le\bar\epsilon_i$. Hence $\psi_i:(M_i,f_i,x_i)\rightarrowtail(\widehat M,\hat x,\hat f)$ is an $(m_i,R'_i,\bar\lambda_i,\bar\epsilon_i)$-pointed local quasi-equivalence. It follows that $([M_i,f_i,x_i],[\widehat M,\hat x,\hat f])\in\widehat U^{m_i}_{R'_i,s_i}$ for any sequence $s_i\downarrow0$ so that $s_i>\max\{\ln\bar\lambda_i,\bar\epsilon_i\}$, obtaining that $[M_i,f_i,x_i]\to[\widehat M,\hat x,\hat f]$ as $i\to\infty$ in $\widehat\MM_*^\infty(n)$.
\end{proof}

\begin{cor}\label{c: widehat MM_*^infty(n) is Polish}
  	$\widehat\MM_*^\infty(n)$ is Polish.
\end{cor}

\begin{proof}
  This is the content of Propositions~\ref{p: widehat MM_*^infty(n) is separable} and~\ref{p: widehat MM_*^infty(n) is completely metrizable} together.
\end{proof}
 
Corollaries~\ref{c: C^infty convergence in MM_*(n)} and~\ref{c: widehat MM_*^infty(n) is Polish} give Theorem~\ref{t: C^infty convergence in widehat MM_*(n)}.

\section{Foliated structure of $\widehat\MM_{*,\text{\rm imm}}^\infty(n)$}\label{s: foliated structure}

The properties stated in Theorem~\ref{t: widehat FF_*,imm(n)} are given by propositions of this section.

\begin{prop}\label{p: widehat MM_*,imm(n) is Polish}
	$\widehat\MM_{*,\text{\rm imm}}^\infty(n)$ is Polish.
\end{prop}

\begin{proof}
	For each $R>0$, let $\WW_R\subset\widehat\MM_*^\infty(n)$ be the open subset consisting of the points $[M,f,x]$ such that $f|_\Omega$ is an immersion for some compact domain $\Omega\subset M$ containing $B_M(x,R)$.  Then $\widehat\MM_{*,\text{\rm imm}}(n)=\bigcap_{R=1}^\infty\WW_R$ is a $G_\delta$ in $\widehat\MM_*^\infty(n)$. So $\widehat\MM_{*,\text{\rm imm}}^\infty(n)$ is a Polish space by Corollary~\ref{c: widehat MM_*^infty(n) is Polish} and \cite[Theorem~I.3.11]{Kechris1995}.
\end{proof}

\begin{prop}\label{p: density}
	$\widehat\MM_{*,\text{\rm imm,c}}^\infty(n)$ is dense in $\widehat\MM_{*,\text{\rm c}}^\infty(n)$.
\end{prop}

\begin{proof}
	With the notation of the proof of Proposition~\ref{p: widehat MM_*^infty(n) is separable}, $\widehat\MM_{*,\text{\rm c}}^\infty(n)$ has an open partition consisting of the subspaces
		\[
			\widehat\MM_*^\infty(M)=\{\,[M,f,x]\mid f\in C^\infty(M,\E),\ x\in M\,\}\quad(M\in\CC)\;.
		\]
	Thus it is enough to prove that each intersection $\widehat\MM_*^\infty(M)\cap\widehat\MM_{*,\text{\rm imm}}^\infty(n)$ is dense in $\widehat\MM_*^\infty(M)$. This means that $C^\infty_{\text{\rm imm}}(M,\E)$ is dense in $C^\infty(M,\E)$, which follows easily from \cite[Theorem~2.2.12]{Hirsch1976}.
\end{proof}

\begin{prop}\label{p: transitivity}
	There is a connected complete open Riemannian manifold $N$ and some $h\in C^\infty_{\text{\rm imm}}(N,\E)$ such that $\hat\iota_{N,h}$ is dense in $\widehat\MM_{*,\text{\rm imm,o}}^\infty(n)$.
\end{prop}

\begin{proof}
	In the proof of Proposition~\ref{p: widehat MM_*^infty(n) is separable}, we can assume that $\FF_M\subset C^\infty_{\text{\rm imm}}(M,\E)$ for each $M\in\CC$ by \cite[Theorem~2.2.12]{Hirsch1976}. Then the set~\eqref{countable dense subset}, denoted here by $\{\,[(M_i,g_i),f_i,x_i]\mid i\in\N\,\}$, is contained in $\widehat\MM_{*,\text{\rm imm}}^\infty(n)$.
	
	For every $i$, let $r_i=\max_{x\in M_i}d(x_i,x)$, and let $B_i=B_i(x_i,r_i/2)$ and $B'_i=B_i(x_i,2r_i/3)$. Let $N$ be a $C^\infty$ connected manifold obtained by modifying $\bigsqcup_iM_i$ on the complement of $\bigsqcup_i\ol{B'_i}$; for instance, we can take $N$ equal to the $C^\infty$ connected sum $M_0\,\#\,M_1\,\#\,\cdots$, constructed by removing balls in the sets $M_i\sm\ol{B'_i}$. Equip $N$ with a complete Riemannian metric $g_N$ whose restriction to each $B_i$ is $g_i$. For instance, we can take $g_N=\lambda g'+\mu g''$, where $\{\lambda,\mu\}$ is a $C^\infty$ partition of unity of $N$ subordinated to the open covering $\{\bigsqcup_iB'_i,N\sm\bigsqcup_i\ol{B_i}\}$, $g'$ is the combination of the metrics $g_i$ on $\bigsqcup_iB'_i$, and $g''$ is any complete metric on $N$. Form \cite[Theorems~2.1.1 and~2.2.12]{Hirsch1976}, it easily follows that there is some $h\in C^\infty_{\text{\rm imm}}(N,\E)$ whose restriction to each $B_i$ is $f_i$. It is easy to see that $N$ and $h$ satisfies the conditions of the statement.
\end{proof}

\begin{rem}
	The versions of Propositions~\ref{p: density} and~\ref{p: transitivity} with embeddings instead of immersions also hold by \cite[Theorems~2.1.4 and~2.2.13]{Hirsch1976}.
\end{rem}

To define foliated charts in $\widehat\MM_{*,\text{\rm imm}}^\infty(n)$, fix some $e\in\E$, and some linear subspace, $V\subset\E$, of dimension $n$. Let $\Pi_V:\E\to V$ denote the orthogonal projection. For each complete connected Riemannian manifold $M$ and any $f\in C^\infty_{\text{\rm imm}}(M,\E)$, let $\chi_{M,f}=\chi_{V,e,M,f}:M\to V$ be the $C^\infty$ map defined by $\chi_{M,f}(x)=\Pi_V(f(x)-e)$. Let $\chi=\chi_{V,e}:\MM_{*,\text{\rm imm}}^\infty(n)\to V$ be defined by $\chi([M,f,x])=\chi_{M,f}(x)$. 

\begin{lem}\label{l: chi is cont}
	$\chi$ is continuous
\end{lem}

\begin{proof}
	The map $\chi$ equals the following composite of continuous maps:
		\begin{equation}\label{chi}
			\begin{CD}
				\MM_{*,\text{\rm imm}}^\infty(n) @>{\ev}>> \E @>{-e}>> \E @>{\Pi_V}>> V\;,
			\end{CD}
		\end{equation}
	where the translation by $-e$ in $\E$ is also denoted by $-e$.
\end{proof}

Given $\rho,\sigma>0$ and $\kappa>1$, let $B=B_V(0,\sigma)$, and consider the following subsets of $\widehat\MM_{*,\text{\rm imm}}^\infty(n)$:
	\begin{itemize}
	
		\item $\NN_0=\NN_0(V,e,\rho,\kappa,\sigma)$ consists of the classes $[M,f,x]\in\widehat\MM_{*,\text{\rm imm}}^\infty(n)$ such that $\chi_{M,f}:B_M(x,\tilde\rho)\to V$ is a $\tilde\kappa$-quasi-isometric embedding for some $\tilde\rho>5\rho+\kappa\sigma$ and $\tilde\kappa\in(1,\kappa)$, and $\ol B\subset\chi_{M,f}(B_M(x,\rho))$. 
		
		\item $\NN_1=\NN_1(V,e,\rho,\kappa,\sigma)$ consists of the classes $[M,f,x]\in\widehat\MM_{*,\text{\rm imm}}^\infty(n)$ such that $[M,f,x']\in\NN_0$ for some $x'\in B_M(x,\rho)$.
		
		\item $\NN_2=\NN_2(V,e,\rho,\kappa,\sigma):=\NN_1\cap\chi^{-1}(B)$.
		
	\end{itemize}
Using \cite[Lemma~2.1.3]{Hirsch1976}, it easily follows that, for each $i\in\{0,1,2\}$, the sets $\NN_i(V,e,\rho,\kappa,\sigma)$ form an open covering of $\MM_{*,\text{\rm imm}}^\infty(n)$ by varying $(V,e,\rho,\kappa,\sigma)$. 

\begin{lem}\label{l: chi_M,f : B_M(x,4 rho + kappa sigma) to V}
	$\chi_{M,f}:B_M(x,4\rho+\kappa\sigma)\to V$ is an embedding for all $[M,f,x]\in\NN_1$.
\end{lem}

\begin{proof}
	For each $[M,f,x]\in\NN_1$, take some $x'\in B_M(x,\rho)$ so that $[M,f,x']\in\NN_0$. Then $B_M(x,4\rho+\kappa\sigma)\subset B_M(x',5\rho+\kappa\sigma)$ and $\chi_{M,f}:B_M(x',5\rho+\kappa\sigma)\to V$ is an embedding.
\end{proof}

Let $\ZZ=\NN_1\cap\chi^{-1}(0)$, which is closed in $\NN_2$.  For each $[M,f,x]\in\NN_1$, there is some $x'\in B_M(x,\rho)$ so that $[M,f,x']\in\NN_0$. Then there is some $x''\in B_M(x',\rho)$ such that $\chi_{M,f}(x'')=0$. Observe that $[M,f,x'']\in\NN_1$, and therefore $[M,f,x'']\in\ZZ$. By Lemma~\ref{l: chi_M,f : B_M(x,4 rho + kappa sigma) to V}, $x''$ is the unique point in $B_M(x,2\rho)$ such that $\chi_{M,f}(x'')=0$. Thus the class $[M,f,x'']$ depends only on $[M,f,x]$. So a map $\Theta:\NN_1\to\ZZ$ is well defined by setting $\Theta([M,f,x])=[M,f,x'']$. 

\begin{lem}\label{l: Theta is cont}
	$\Theta$ is continuous.
\end{lem}

\begin{proof}
	Consider a convergent sequence $[M_i,f_i,x_i]\to[M,f,x]$ in $\NN_1$. Take points $x_i'\in B_i(x_i,2\rho)$ and $x'\in B_M(x',2\rho)$ such that $\chi_{M_i,f_i}(x_i')=\chi_{M,f}(x')=0$. Thus $\Theta([M_i,f_i,x_i])=[M_i,f_i,x_i']$ and $\Theta([M,f,x])=[M,f,x']$.
	
	Given $m\in\N$ and $R,r>0$, for $i$ large enough, there is an $(m,R,\lambda_i,\epsilon_i)$-pointed local quasi-equivalence $\phi_i:(M,f,x)\rightarrowtail(M_i,f_i,x_i)$ for some $\lambda_i\in(1,e^r)$ and $\epsilon_i\in(0,r)$. Suppose that $R>3\rho$ and $e^r<3/2$; in particular, $\ol B_M(x,3\rho)\subset\dom\phi_i$. 
	
	\begin{claim}\label{cl: B_i(x_i, 2 rho_0) subset phi_i(B_M(x,rho))}
		$B_i(x_i,2\rho)\subset\phi_i(B_M(x,3\rho))$.
	\end{claim}
	
	The set $A=B_i(x_i,2\rho)\cap\phi_i(B_M(x,3\rho))$ contains $x_i$ and is open in the connected space $B_i(x_i,2\rho)$. Then Claim~\ref{cl: B_i(x_i, 2 rho_0) subset phi_i(B_M(x,rho))} follows by showing that $A$ is also closed in $B_i(x_i,2\rho)$. This holds since $A=B_i(x_i,2\rho)\cap\phi_i(\ol B_M(x,3\rho))$ because, for every $y\in M$ with $d_M(x,y)=3\rho$, we have
		\[
			d_i(x_i,\phi_i(y))\ge\frac{1}{\lambda_i}d_M(x,y)>3\rho e^{-r}>2\rho\;.
		\]
	
	According to Claim~\ref{cl: B_i(x_i, 2 rho_0) subset phi_i(B_M(x,rho))}, there is some $\bar x_i'\in B_M(x,3\rho)$ such that $\phi_i(\bar x_i')=x_i'$. We have
		\begin{multline*}
			d_M(x',\bar x_i')\le\kappa\|\chi_{M,f}(x')-\chi_{M,f}(\bar x_i')\|
			=\kappa\|\chi_{M,f}(\bar x_i')-\chi_{M_i,f_i}(x_i')\|\\
			\le\kappa\|f(\bar x_i')-f_i(x_i')\|=\kappa\|f(\bar x_i')-f_i\circ\phi(\bar x_i')\|
			<\kappa\epsilon_i<\kappa r\;.
		\end{multline*}
	Therefore, by the continuity of $\hat\iota_{M,f}$, for any $S,s>0$, if $r$ is small enough and $i$ large enough, there is an $(m,S,\mu_i,\delta_i)$-pointed local quasi-equivalence $\psi_i:(M,f,x')\rightarrowtail(M,f,\bar x_i')$ with $\mu_i\in(1,e^{s/2})$ and $\delta_i\in(0,s/2)$. On the other hand, observe that $\phi_i:(M,\bar x_i',f)\rightarrowtail(M_i,f_i,x_i')$ is an $(m,R-2\rho,\lambda_i,\epsilon_i)$-pointed local quasi-equivalence. Hence, if moreover $R>e^{s/2}S+2\rho$ and $r<s/2$, we get that $\phi_i\circ\psi_i:(M,f,x')\rightarrowtail(M_i,f_i,x_i')$ is an $(m,S,\mu_i\lambda_i,\delta_i+\epsilon_i)$-pointed local quasi-equivalence with $\mu_i\lambda_i\in(1,e^s)$ and $\delta_i+\epsilon_i\in(0,s)$ by Lemma~\ref{l: composition and inversion of pointed local quasi-equivalences}-\eqref{i: composition}. This shows that $[M_i,f_i,x_i']\to[M,f,x']$ in $\widehat\MM_*^\infty(n)$.
\end{proof}

Let $\Phi=(\chi,\Theta):\NN_2\to B\times\ZZ$.

\begin{lem}\label{l: Phi^-1}
	$\Phi$ is bijective, and $\Phi^{-1}(v,[M,f,x])=[M,f,x']$ for each $(v,[M,f,x])\in B\times\ZZ$, where $x'$ is the unique point in $B_M(x,2\rho)\cap\chi_{M,f}^{-1}(v)$.
\end{lem}

\begin{proof}
	To prove that $\Phi$ is injective, let $[M_i,f_i,x_i]\in\NN_2$ ($i\in\{1,2\}$) such that $\Phi([M_1,f_1,x_1])=\Phi([M_2,f_2,x_2])$; i.e., $\chi_{M_1,f_1}(x_1)=\chi_{M_2,f_2}(x_2)$ and $[M_1,f_1,x_1']=[M_2,f_2,x_2']$ for points $x_i'\in B_i(x_i,2\rho)$ with $\chi_{M_i,f_i}(x_i')=0$. Thus there is a pointed equivalence $\phi:(M_1,f_1,x_1')\to(M_2,f_2,x_2')$. We get $\phi(x_1)=x_2$ because $\chi_{M_2,f_2}\circ\phi(x_1)=\chi_{M_1,f_1}(x_1)=\chi_{M_2,f_2}(x_2)$, the map $\chi_{M_i,f_i}:(B_i(x_i',2\rho),x_i')\to(V,0)$ is a pointed embedding (Lemma~\ref{l: chi_M,f : B_M(x,4 rho + kappa sigma) to V}), and $x_i\in B_i(x_i',2\rho)$. So $\phi:(M_1,f_1,x_1)\to(M_2,f_2,x_2)$ is a pointed equivalence, and therefore $[M_1,f_1,x_1]=[M_2,f_2,x_2]$.
	
	Now, let us prove that $\Phi$ is surjective, showing the stated expression of $\Phi^{-1}$. Let $(v,[M,f,x])\in B\times\ZZ$. There is some $y\in B_M(x,\rho)$ such that $[M,f,y]\in\NN_0$. So there is some $x'\in B_M(y,\rho)$ such that $\chi_{M,f}(x')=v$. It follows that $[M,f,x']\in\NN_1$, $\Theta([M,f,x'])=[M,f,x]$ and $\chi([M,f,x'])=v$. Therefore $[M,f,x']\in\NN_2$ and $\Phi([M,f,x'])=(v,[M,f,x])$. Moreover $x'$ is the unique point in $B_M(x,2\rho)\cap\chi_{M,f}^{-1}(v)$ by Lemma~\ref{l: chi_M,f : B_M(x,4 rho + kappa sigma) to V}.
\end{proof}

\begin{lem}\label{l: Phi^-1 is cont}
	$\Phi^{-1}$ is continuous.
\end{lem}

\begin{proof}
	Consider a convergent sequence $(v_i,[M_i,f_i,x_i])\to(v,[M,f,x])$ in $B\times\ZZ$. Take points $x_i'\in B_i(x_i,2\rho)$ and $x'\in B_M(x,2\rho)$ such that $\chi_{M_i,f_i}(x_i')=v_i$ and $\chi_{M,f}(x')=v$. Thus $\Phi^{-1}(v_i,[M_i,f_i,x_i])=[M_i,f_i,x_i']$ and $\Phi^{-1}(v,[M,f,x])=[M,f,x']$.
	
	Given $m\in\N$ and $R,r>0$, if $i$ is large enough, then $\|v-v_i\|<r$, and there is an $(m,R,\lambda_i,\epsilon_i)$-pointed local quasi-equivalence $\phi_i:(M,f,x)\rightarrowtail(M_i,f_i,x_i)$ for some $\lambda_i\in(1,e^r)$ and $\epsilon_i\in(0,r)$. Suppose that $R>3\rho$ and $e^r<3/2$; in particular, $\ol B_M(x,3\rho)\subset\dom\phi_i$. Like in Claim~\ref{cl: B_i(x_i, 2 rho_0) subset phi_i(B_M(x,rho))}, we get $B_i(x_i,2\rho)\subset\phi_i(B_M(x,3\rho))$. Then, since $x_i'\in B_i(x_i,2\rho)$, there is some $\bar x_i'\in B_M(x,3\rho)$ such that $\phi_i(\bar x_i')=x_i'$. We have
		\begin{multline*}
			d_M(x',\bar x_i')\le\kappa\|\chi_{M,f}(x')-\chi_{M,f}(\bar x_i')\|
			\le\kappa\left(\|\chi_{M,f}(\bar x_i')-\chi_{M_i,f_i}(x_i')\|+\|v-v_i\|\right)\\
			<\kappa\left(\|f(\bar x_i')-f_i(x_i')\|+r\right)=\kappa\left(\|f(\bar x_i')-f_i\circ\phi(\bar x_i')\|+r\right)
			<\kappa(\epsilon_i+r)<2\kappa r\;.
		\end{multline*}
	Hence we get $[M_i,f_i,x_i']\to[M,f,x']$ in $\widehat\MM_*^\infty(n)$ like in the end of the proof of Lemma~\ref{l: Theta is cont}.
\end{proof}

\begin{cor}\label{c: Phi is a homeomorphism}
	$\Phi$ is a homeomorphism.
\end{cor}

\begin{proof}
	This follows from Lemmas~\ref{l: chi is cont},~\ref{l: Theta is cont},~\ref{l: Phi^-1} and~\ref{l: Phi^-1 is cont}.
\end{proof}

\begin{lem}\label{l: NN_2}
	If $[M,f,x]\in\chi^{-1}(B)$ and $[M,f,x']\in\ZZ$ for some $x'\in B_M(x,2\rho)$, then $[M,f,x]\in\NN_2$.
\end{lem}

\begin{proof}
	 Let $v=\chi([M,f,x])\in B$. By Lemma~\ref{l: Phi^-1}, there is some $x''\in B_M(x',2\rho)$ such that $[M,f,x'']\in\NN_2$ and $\Phi([M,f,x''])=(v,[M,f,x'])$. Then $x=x''$ by Lemma~\ref{l: chi_M,f : B_M(x,4 rho + kappa sigma) to V} applied to $\chi_{M,f}:B_M(x',2\rho)\to V$.
\end{proof}

Take $(\widetilde V,\tilde e,\tilde\rho,\tilde\kappa,\tilde\sigma)$ like $(V,e,\rho,\kappa,\sigma)$. Let $\widetilde\NN_i=\NN_i(\widetilde V,\tilde e,\tilde\rho,\tilde\kappa,\tilde\sigma)$ for $i\in\{0,1,2\}$, and let $\widetilde\Phi=(\widetilde\chi,\widetilde\Theta):\widetilde\NN_2\to\widetilde B\times\widetilde\ZZ$ be defined like $\Phi=(\chi,\Theta):\NN_2\to B\times\ZZ$, using $(\widetilde V,\tilde e,\tilde\rho,\tilde\kappa,\tilde\sigma)$. Moreover, for each $[M,f,x]\in\widehat\MM_*^\infty(n)$, let $\widetilde\chi_{M,f}:M\to\widetilde V$ be defined like $\chi_{M,f}:M\to V$, using $\Pi_{\widetilde V}$ and $\tilde e$. Suppose that $\NN_2\cap\widetilde\NN_2\ne\emptyset$, and consider the map $\widetilde\Phi\circ\Phi^{-1}:\Phi(\NN_2\cap\widetilde\NN_2)\to\widetilde\Phi(\NN_2\cap\widetilde\NN_2)$.

\begin{lem}\label{l: widetilde Phi circ Phi^-1}
	Let $(v,[M,f,x])\in\Phi(\NN_2\cap\widetilde\NN_2)$. Then $\widetilde\Phi\circ\Phi^{-1}(v,[M,f,x])=(\tilde v,[M,f,\tilde x])$, where $\tilde x\in\widetilde\chi_{M,f}^{-1}(0)$ is determined by the condition
		\begin{equation}\label{[M,tilde x,f]}
			B_M(x,2\rho)\cap B_M(\tilde x,2\tilde\rho)
			\cap\chi_{M,f}^{-1}(B)\cap\widetilde\chi_{M,f}^{-1}(\widetilde B)\ne\emptyset\;,
		\end{equation}
	and $\tilde v$ is the image of $v$ by the composite
		\begin{equation}\label{tilde v}
			\begin{CD}
				\chi_{M,f}(O) @>{\chi_{M,f}^{-1}}>> O @>{\widetilde\chi_{M,f}}>> \widetilde\chi_{M,f}(O)\;,
			\end{CD}
		\end{equation}
	where $O=B_M(x,2\rho)\cap B_M(\tilde x,2\tilde\rho)$.
\end{lem}

\begin{proof}
	Let $[M,f,x']\in\NN_2\cap\widetilde\NN_2$ such that $\Phi([M,f,x'])=(v,[M,f,x])$ and $\widetilde\Phi([M,f,x'])=(\tilde v,[M,f,\tilde x])$. By Lemma~\ref{l: Phi^-1}, this means that $\chi_{M,f}(x)=\widetilde\chi_{M,f}(\tilde x)=0$, $x'\in B_M(x,2\rho)\cap B_M(\tilde x,2\tilde\rho)$, $\chi_{M,f}(x')=v$ and $\widetilde\chi_{M,f}(x')=\tilde v$, obtaining~\eqref{[M,tilde x,f]} and~\eqref{tilde v}. Note that~\eqref{tilde v} makes sense by Lemma~\ref{l: chi_M,f : B_M(x,4 rho + kappa sigma) to V}.
	
	Now, assume that~\eqref{[M,tilde x,f]} also holds using another point $\tilde y\in\widetilde\chi_{M,f}^{-1}(0)$ instead of $\tilde x$. Thus there is some $y'\in B_M(x,2\rho)\cap B_M(\tilde x,2\tilde\rho)$ with $w:=\chi_{M,f}(y')\in B$ and $\tilde w:=\widetilde\chi_{M,f}(y')=\widetilde B$. Then $[M,f,y']\in\NN_2$ by Lemma~\ref{l: NN_2}, and $\Phi([M,f,y'])=(w,[M,f,x])$ and $\widetilde\Phi([M,f,y'])=(w,[M,f,\tilde y])$. We have
		\[
			d_M(\tilde x,\tilde y)\le d_M(\tilde x,x')+d_M(x',y')+d_M(y',\tilde y)
			<4\tilde\rho+\tilde\kappa\|\tilde v-\tilde w\|<4\tilde\rho+\tilde\kappa\tilde\sigma\;.
		\]
	Since moreover $\tilde\chi_{M,f}(\tilde x)=0=\tilde\chi_{M,f}(\tilde y)$, we get $\tilde x=\tilde y$ by Lemma~\ref{l: chi_M,f : B_M(x,4 rho + kappa sigma) to V}.
\end{proof}

\begin{prop}\label{p: fol str}
	All possible maps $\Phi:\NN_2\to B\times\ZZ$ form an atlas of a $C^\infty$ foliated structure on $\widehat\MM_{*,\text{\rm imm}}^\infty(n)$.
\end{prop}

\begin{proof}
	The maps $\Phi:\NN_2\to B\times\ZZ$ are homeomorphisms (Corollary~\ref{c: Phi is a homeomorphism}). All possible sets $\NN_2$ form an open cover of $\widehat\MM_{*,\text{\rm imm}}^\infty(n)$. Moreover, in Lemma~\ref{l: widetilde Phi circ Phi^-1}, it follows from~\eqref{[M,tilde x,f]} that $[M,f,\tilde x]$ depends only on $[M,f,x]$. Thus all possible maps $\Phi:\NN_2\to B\times\ZZ$ form an atlas of a foliated structure on $\widehat\MM_{*,\text{\rm imm}}^\infty(n)$.
	
	With the notation of Lemma~\ref{l: widetilde Phi circ Phi^-1} and the terminology of \cite[Section~2.1]{AlvarezBarralCandel2016}, it only remains to show that $\widetilde\Phi\circ\Phi^{-1}$ is $C^\infty$; i.e., to prove that the mapping $(v,[M,f,x])\mapsto\tilde v$ is $C^\infty$. First, note that, for each fixed $[M,f,x]$, the mapping $v\mapsto\tilde v$ is $C^\infty$ because~\eqref{tilde v} is $C^\infty$. Consider now a convergent sequence $[M_i,f_i,x_i]\to[M,f,x]$ in $\ZZ$. Let $\tilde x_i\in\widetilde\chi_{M_i,f_i}^{-1}(0)$ be determined by
		\[
			B_M(x_i,2\rho)\cap B_M(\tilde x_i,2\tilde\rho)
			\cap\chi_{M_i,f_i}^{-1}(B)\cap\widetilde\chi_{M_i,f_i}^{-1}(\widetilde B)\ne\emptyset\;,
		\]
	and let $O_i=B_M(x_i,2\rho)\cap B_M(\tilde x,2\tilde\rho)$. Given $m\in\N$ and $R,r>0$, for each $i$ large enough, there is an $(m,R,\lambda_i,\epsilon_i)$-pointed local quasi-equivalence $\phi_i:(M_i,f_i,x_i)\rightarrowtail(M,f,x)$ for some $\lambda_i\in(1,e^r)$ and $\epsilon_i\in(0,r)$. Let $\Omega_i^{(m)}$ be a compact domain in $\dom\phi_{i*}^{(m)}$ such that $B_i^{(m)}(x_i,R)\subset\Omega_i^{(m)}$ and $\phi_{i*}^{(m)}:(\Omega_i^{(m)},f_{i*}^{(m)})\to(T^{(m)}M,f_*^{(m)})$ is an $(\epsilon_i,\lambda_i)$-quasi-equivalence. Since $(\Pi_V)^{(m)}_*\equiv\Pi_{V^{2^m}}:T^{(m)}\E\equiv\E^m\to T^{(m)}V\equiv V^{2^m}$, we have
		\begin{multline}\label{|chi_M,f*^(m) - chi_M_i,f_i*^(m) circ phi_i*^(m)|_Omega_i^(m)}
			\left\|\chi_{M_i,f_i*}^{(m)}-(\chi_{M,f}\circ\phi_{i})_*^{(m)}\|_{\Omega_i^{(m)}}
			\le\|(f_i-e)_*^{(m)}-((f-e)\circ\phi_{i})_*^{(m)}\right\|_{\Omega_i^{(m)}}\\
			=\left\|f_{i*}^{(m)}-(\phi_i^*f)_*^{(m)}\right\|_{\Omega_i^{(m)}}<\epsilon_i<r\;.
		\end{multline}
	
	Assume that $R>2e^r\rho$. Then $B_i(x_i,2\rho)\subset B_i(x_i,R)$, and, like in Claim~\ref{cl: B_i(x_i, 2 rho_0) subset phi_i(B_M(x,rho))}, we also get $B_M(x,2\rho)\subset\phi_i(B_i(x_i,R))$. Thus $O_i\subset B_i(x_i,R)$ and $O\subset\phi_i(B_i(x_i,R))$. Let $\Xi\subset\chi_{M,f}(O)$ be a compact domain, which is also contained in $\chi_{M_i,f_i}(O_i)$ for $i$ large enough. Let $\Xi^{(m)}$ be a compact domain contained in $T^{(m)}\E$ such that
		\[
			\Xi\subset\Int(\Xi^{(m)})\;,\quad
			(\chi_{M_i,f_i}^{-1})_*^{(m)}(\Xi^{(m)})\subset\Omega_i^{(m)}\cap T^{(m)}O_i\;,\quad
			(\chi_{M,f}^{-1})_*^{(m)}(\Xi^{(m)})\subset\phi_{i*}^{(m)}(\Omega_i^{(m)})\cap T^{(m)}O\;.
		\]
	Since the restrictions of $(\chi_{M,f}^{-1})_*^{(m)}$ and $(\widetilde\chi_{M,f})_*^{(m)}$ to the respective compact domains $\Xi^{(m)}$ and $\chi_{M,f}^{-1}(\Xi^{(m)})\cap T^{(m)}O$ are $C^\infty$ embeddings, these restrictions are $\nu$-quasi-isometric for some $\nu\ge1$. Hence, by \eqref{|chi_M,f*^(m) - chi_M_i,f_i*^(m) circ phi_i*^(m)|_Omega_i^(m)},
		\begin{multline}\label{d_i^(m)((phi_i circ chi_M_i,f_i^-1)_*^(m)(xi),(chi_M,f^-1)_*^(m)(xi))}
			d_i^{(m)}\left((\phi_i\circ\chi_{M_i,f_i}^{-1})_*^{(m)}(\xi),(\chi_{M,f}^{-1})_*^{(m)}(\xi)\right)
			\le\nu\left\|(\chi_{M,f}\circ\phi_i\circ\chi_{M_i,f_i}^{-1})_*^{(m)}(\xi)-\xi\right\|\\
			=\nu\left\|(\chi_{M,f}\circ\phi_i\circ\chi_{M_i,f_i}^{-1})_*^{(m)}(\xi)
			-(\chi_{M_i,f_i}\circ\chi_{M_i,f_i}^{-1})_*^{(m)}(\xi)\right\|
			<\nu r\;,
		\end{multline}
	for all $\xi\in\Xi^{(m)}$. On the other hand, like in~\eqref{|chi_M,f*^(m) - chi_M_i,f_i*^(m) circ phi_i*^(m)|_Omega_i^(m)}, we get
		\begin{equation}\label{|widetilde chi_M,f*^(m) - widetilde chi_M_i,f_i*^(m) circ phi_i*^(m)|_Omega_i^(m)}
			\left\|\widetilde\chi_{M_i,f_i*}^{(m)}
			-(\widetilde\chi_{M,f}\circ\phi_{i})_*^{(m)}\right\|_{\Omega_i^{(m)}}<r\;.
		\end{equation}
	Combining~\eqref{d_i^(m)((phi_i circ chi_M_i,f_i^-1)_*^(m)(xi),(chi_M,f^-1)_*^(m)(xi))} and~\eqref{|widetilde chi_M,f*^(m) - widetilde chi_M_i,f_i*^(m) circ phi_i*^(m)|_Omega_i^(m)}, we obtain the following for all $\xi\in\Xi^{(m)}$:
		\begin{multline*}
			\left\|(\widetilde\chi_{M_i,f_i}\circ\chi_{M_i,f_i}^{-1})_*^{(m)}(\xi)
			-(\widetilde\chi_{M,f}\circ\chi_{M,f}^{-1})_*^{(m)}(\xi)\right\|\\
			\begin{aligned}
				&\le\left\|(\widetilde\chi_{M_i,f_i}\circ\chi_{M_i,f_i}^{-1})_*^{(m)}(\xi)
				-(\widetilde\chi_{M,f}\circ\phi_i\circ\chi_{M_i,f_i}^{-1})_*^{(m)}(\xi)\right\|\\
				&\phantom{\le\text{}}
				\text{}+\left\|(\widetilde\chi_{M,f}\circ\phi_i\circ\chi_{M_i,f_i}^{-1})_*^{(m)}(\xi)
				-(\widetilde\chi_{M,f}\circ\chi_{M,f}^{-1})_*^{(m)}(\xi)\right\|
			\end{aligned}\\
			<r+\nu\,d_i^{(m)}\left((\phi_i\circ\chi_{M_i,f_i}^{-1})_*^{(m)}(\xi),(\chi_{M,f}^{-1})_*^{(m)}(\xi)\right)
			<(1+\nu^2)r\;.
		\end{multline*}
	Note that the same choices of $\Xi$ and $\Xi^{(m)}$ are valid for all $r$ small enough, obtaining that $(\widetilde\chi_{M_i,f_i}\circ\chi_{M_i,f_i}^{-1})_*^{(m)}\to(\widetilde\chi_{M,f}\circ\chi_{M,f}^{-1})_*^{(m)}$ uniformly on $\Xi^{(m)}$. Moreover the same choice of $\Xi$ is valid for all $m$, and therefore $\widetilde\chi_{M_i,f_i}\circ\chi_{M_i,f_i}^{-1}\to\widetilde\chi_{M,f}\circ\chi_{M,f}^{-1}$ on $\Xi$ with respect to the $C^\infty$ topology by the obvious version of Lemma~\ref{l: f_*^(m)} for maps between open subsets of $\R^n$. Since every point in $\chi_{M,f}(O)$ belongs to some domain $\Xi$ as above if $r$ is chosen small enough, it follows that $\widetilde\Phi\circ\Phi^{-1}$ is $C^\infty$.
\end{proof}

Now, let $\widehat\FF_{*,\text{\rm imm}}^\infty(n)$ denote the $C^\infty$ foliated structure on $\widehat\MM_{*,\text{\rm imm}}^\infty(n)$ defined by the maps $\Phi$ according to Proposition~\ref{p: fol str}.

\begin{prop}\label{p: widehat FF_*,imm^infty(n)}
	The following properties hold:
		\begin{enumerate}[{\rm(}i\/{\rm)}]
	
			\item\label{i: uniqueness of widehat FF_*,imm^infty(n)} $\widehat\FF_{*,\text{\rm imm}}^\infty(n)$ is the unique $C^\infty$ foliated structure on $\widehat\MM_{*,\text{\rm imm}}^\infty(n)$ such that its underlying topological foliated structure is $\widehat\FF_{*,\text{\rm imm}}(n)$ and $\ev:\widehat\MM_{*,\text{\rm imm}}^\infty(n)\to\E$ is a $C^\infty$ immersion. 
	
	 		\item\label{i: hat iota_M,f is a local diffeomorphism} For each $[M,f,x]\in\widehat\MM_{*,\text{\rm imm}}^\infty(n)$, the map $\hat\iota_{M,f}:M\to\im\hat\iota_{M,f}$ is a local diffeomorphism, where the leaf $\im\hat\iota_{M,f}$ is equipped with the $C^\infty$ structure induced by $\widehat\FF_{*,\text{\rm imm}}^\infty(n)$.
		
		\end{enumerate}
\end{prop}

\begin{proof}
	Take a foliated chart $\Phi:\NN_2\to B\times\ZZ$ as above. For each $[M,f,x]\in\ZZ$, the restriction of $\ev\circ\Phi^{-1}$ to $B\times\{[M,f,x]\}\equiv B$ is the composite
		\[
			\begin{CD}
				B @>{\chi_{M,f}^{-1}}>> B_M(x,2\rho)\cap\chi_{M,f}^{-1}(B) @>f>> \E\;,
			\end{CD}
		\]
	where the first map is a $C^\infty$ diffeomorphism, and the second one is a $C^\infty$ immersion. Take a convergent sequence $[M_i,f_i,x_i]\to[M,f,x]$ in $\ZZ$, and let $\Xi\subset B$ be any compact domain. Given $R>2\rho$ and a compact domain $\Omega\subset M$ containing $B_M(x,R)$, there is a $C^\infty$ pointed embedding $\phi_i:(\Omega,x)\to(M_i,x_i)$ for $i$ large enough such that $\phi_i^*g_i\to g_M$ and $\phi_i^*f_i\to f$ on $\Omega$ with respect to the $C^\infty$ topology. So $B_i(x_i,R)\subset\phi_i(\Omega)$ for $i$ large enough. Thus also $\phi_i^*\chi_{M_i,f_i}\to\chi_{M,f}$ on $\Omega$ with respect to the $C^\infty$ topology, and therefore $\phi_i^{-1}\circ\chi_{M_i,f_i}^{-1}\to\chi_{M,f}^{-1}$ on $\Xi$ with respect to the $C^\infty$ topology \cite[p.~64, Exercise~9]{Hirsch1976}. Hence
		\[
			f_i\circ\chi_{M_i,f_i}^{-1}-f\circ\chi_{M,f}^{-1}
			=f_i\circ\phi_i\circ(\phi_i^{-1}\circ\chi_{M_i,f_i}^{-1}-\chi_{M,f}^{-1})
			+(f_i\circ\phi_i-f)\circ\chi_{M,f}^{-1}\to0
		\]
	on $\Xi$ with respect to the $C^\infty$ topology. Since any element of $B$ is contained some $\Xi$ as above, it follows that $\ev\circ\Phi^{-1}$ is a $C^\infty$ immersion, and therefore $\ev:\widehat\MM_{*,\text{\rm imm}}^\infty(n)\to\E$ is $C^\infty$ with respect to $\widehat\MM_{*,\text{\rm imm}}^\infty(n)$. This shows~\eqref{i: uniqueness of widehat FF_*,imm^infty(n)}, except uniqueness.
	
	According to Lemma~\ref{l: Phi^-1}, for each chart $\Phi:\NN_2\to B\times\ZZ$, the plaque that corresponds to each $[M,f,x]\in\ZZ$ is $\hat\iota_{M,f}(B_M(x,2\rho)\cap\chi_{M,f}^{-1}(B))$. Moreover the composite
		\[
			\begin{CD}
				B_M(x,2\rho)\cap\chi_{M,f}^{-1}(B) @>{\hat\iota_{M,f}}>> 
				\hat\iota_{M,f}(B_M(x,2\rho)\cap\chi_{M,f}^{-1}(B)) @>{\chi}>> B
			\end{CD}
		\]
	is the diffeomorphism $\chi_{M,f}:B_M(x,2\rho)\cap\chi_{M,f}^{-1}(B)\to B$. This shows that the leaf topology on $\widehat\MM_{*,\text{\rm imm}}(n)$ equals the topological sum of all possible spaces $\im\hat\iota_{M,f}$ with the topology so that $\hat\iota_{M,f}:M\to\im\hat\iota_{M,f}$ is a local homeomorphism, obtaining that these spaces are the leaves because they are connected. It also follows that $\hat\iota_{M,f}:M\to\im\hat\iota_{M,f}$ is a local diffeomorphism for each leaf $\im\hat\iota_{M,f}$. This shows~\eqref{i: hat iota_M,f is a local diffeomorphism}.
	
	Now, suppose $\ev:\widehat\MM_{*,\text{\rm imm}}^\infty(n)\to\E$ is $C^\infty$ with respect to some $C^\infty$ foliated structure $\GG$ whose underlying topological foliated structure is $\widehat\FF_{*,\text{\rm imm}}(n)$. Then $\chi:\widehat\MM_{*,\text{\rm imm}}^\infty(n)\to V$ is also $C^\infty$ with respect to $\GG$ because it equals the composite~\eqref{chi}. So each chart $\Phi=(\chi,\Theta):\NN_2\to B\times\ZZ$ is also $C^\infty$ with respect to $\GG$ and the $C^\infty$ product foliated structure of $B\times\ZZ$. Moreover, for all complete connected Riemannian manifold $M$ and $f\in C^\infty_{\text{\rm imm}}(M)$, the map $\hat\iota_{M,f}:M\to\im\hat\iota_{M,f}$ is a $C^\infty$ local diffeomorphism with respect to the $C^\infty$ structure induced by $\GG$ on the leaf $\im\hat\iota_{M,f}$ because $\ev$ is a $C^\infty$ immersion and $\ev\circ\hat\iota_{M,f}=f$, which is a $C^\infty$ local embedding. Thus the restriction of $\chi:\NN_2\to B$ to each plaque is a $C^\infty$ diffeomorphism. Using again \cite[p.~64, Exercise~9]{Hirsch1976}, it follows that $\Phi:\NN_2\to B\times\ZZ$ is also $C^\infty$ foliated diffeomorphism with respect to the restriction of $\GG$ and the $C^\infty$ product foliated structure of $B\times\ZZ$. This shows that $\GG=\widehat\FF_{*,\text{\rm imm}}^\infty(n)$, completing the proof of~\eqref{i: uniqueness of widehat FF_*,imm^infty(n)}. 
\end{proof}

Consider a leaf $\im\hat\iota_{M,f}$ of $\widehat\FF_{*,\text{\rm imm}}^\infty(n)$. Every $x\in M$ has an open neighborhood $U$ in $M$ so that $f:U\to\E$ is an embedding, obtaining that $\phi(U)\cap U=\emptyset$ for all $\phi\in\Iso(M,f)\sm\{\id_M\}$. Therefore the subgroup $\Iso(M,f)\subset\Iso(M)$ is discrete, the quotient projection $M\to\Iso(M,f)\backslash M$ is a covering map, and there is a unique Riemannian structure on the manifold $\Iso(M,f)\backslash M$ so that $M\to\Iso(M,f)\backslash M$ is a local isometry. Moreover $\hat\iota_{M,f}:M\to\im\hat\iota_{M,f}$ induces a diffeomorphism $\bar\iota_{M,f}:\Iso(M,f)\backslash M\to\im\hat\iota_{M,f}$. Thus $\hat\iota_{M,f}:M\to\im\hat\iota_{M,f}$ is a covering map, and $\im\hat\iota_{M,f}$ has a unique Riemannian metric so that $\hat\iota_{M,f}:M\to\im\hat\iota_{M,f}$ is a local isometry, and therefore $\bar\iota_{M,f}:\Iso(M,f)\backslash M\to\im\hat\iota_{M,f}$ becomes an isometry.

\begin{prop}\label{p: Riem met}
	The above Riemannian metrics on the leaves of $\widehat\FF_{*,\text{\rm imm}}^\infty(n)$ form a $C^\infty$ Riemannian metric on $(\widehat\MM_{*,\text{\rm imm}}^\infty(n),\widehat\FF_{*,\text{\rm imm}}^\infty(n))$.
\end{prop}

\begin{proof}
	Let $\Phi=(\chi,\Theta):\NN_2\to B\times\ZZ$ be defined by any choice of $(V,e,\rho,\kappa,\sigma)$ as above, and let $[M_i,f_i,x_i]\to[M,f,x]$ be a convergent sequence in $\ZZ$. Let $\bar g_M$ and $\bar g_i$ be the metrics on $B$ that correspond to $g_M$ and $g_i$ by the diffeomorphisms
		\[
			\chi_{M,f}:P:=B_M(x,2\rho)\cap\chi_{M,f}^{-1}(B)\to B\;,\quad
			\chi_{M_i,f_i}:P_i:=B_i(x_i,2\rho)\cap\chi_{M_i,f_i}^{-1}(B)\to B\;,
		\]
	respectively (see Lemma~\ref{l: Phi^-1}). According to the proof of Proposition~\ref{p: widehat FF_*,imm^infty(n)}-\eqref{i: hat iota_M,f is a local diffeomorphism}, we have to prove that $\bar g_i\to\bar g_M$ as $i\to\infty$ in the weak $C^\infty$ topology.
	
	Given $m\in\N$, $R,r>0$, for each $i$ large enough, there is an $(m,R,\lambda_i,\epsilon_i)$-pointed local quasi-equivalence $\phi_i:(M,f,x)\rightarrowtail(M_i,f_i,x_i)$ for some $\lambda_i\in(1,e^r)$ and $\epsilon_i\in(0,r)$. Assuming $R>2e^r\rho$, we get $B_M(x,2\rho)\subset B_M(x,R)$ and $B_i(x_i,2\rho)\subset\phi_i(B_M(x,R))$, like in the proof of Proposition~\ref{p: fol str}. Take a compact domain $\Omega_i^{(m)}\subset\dom\phi_{i*}^{(m)}$ such that $B_i^{(m)}(x_i,R)\subset\Omega_i^{(m)}$ and $\phi_{i*}^{(m)}:\Omega_i^{(m)}\to T^{(m)}M$ is a $(\lambda_i,\epsilon_i)$-quasi-isometry. Let $\Xi\subset B$ be a compact domain, and let $\Xi^{(m)}$ be a compact domain contained in $T^{(m)}B$ such that
		\[
			\Xi\subset\Int(\Xi^{(m)})\;,\quad
			(\chi_{M,f}^{-1})_*^{(m)}(\Xi^{(m)})\cap T^{(m)}P
			\subset\Omega_i^{(m)}\;,\quad
			(\chi_{M_i,f_i}^{-1})_*^{(m)}(\Xi^{(m)})\cap T^{(m)}P_i
			\subset\phi_{i*}^{(m)}(\Omega_i^{(m)})\;.
		\]
	Like in~\eqref{d_i^(m)((phi_i circ chi_M_i,f_i^-1)_*^(m)(xi),(chi_M,f^-1)_*^(m)(xi))}, there is some $\nu\ge1$, independent of $i$, such that
		\[
			d_i^{(m)}\left((\phi_i^{-1}\circ\chi_{M_i,f_i}^{-1})_*^{(m)}(\xi),(\chi_{M,f}^{-1})_*^{(m)}(\xi)\right)
			<\nu r\;,
		\]
	for all $\xi\in\Xi^{(m)}$. Since the choice of $\Xi^{(m)}$ is valid for all $r$ small enough, it follows that $\phi_i^{-1}\circ\chi_{M_i,f_i}^{-1}\to\chi_{M,f}^{-1}$ in $C^m(\Xi,M)$ by the obvious version of Lemma~\ref{l: f_*^(m)} for maps between manifolds. Since the choice of $\Xi$ is valid for all $m$, it follows that this convergence also holds in $C^\infty(\Xi,M)$. Take a compact domain $\Omega\subset M$ such that $B_M(x,R)\subset\Omega$ and $\phi_i^*g_i\to g_M$ on $\Omega$ with respect to the $C^\infty$ topology. We get
		\[
			(\phi_i^{-1}\circ\chi_{M_i,f_i}^{-1})^*(\phi_i^*g_i-g_M)\to(\chi_{M,f}^{-1})^*0=0
		\]
	on $\Xi$ with respect to the $C^\infty$ topology. So
		\begin{multline*}
			\bar g_i-\bar g_M=(\chi_{M_i,f_i}^{-1})^*g_i-(\chi_{M,f}^{-1})^*g_M\\
			=(\phi_i^{-1}\circ\chi_{M_i,f_i}^{-1})^*(\phi_i^*g_i-g_M)
			+(\phi_i^{-1}\circ\chi_{M_i,f_i}^{-1})^*g_M-(\chi_{M,f}^{-1})^*g_M
			\to0
		\end{multline*}
	on $\Xi$ with respect to the $C^\infty$ topology. Since every point in $B$ belongs to some domain $\Xi$ as above if $r$ is chosen small enough, it follows that $\bar g_i-\bar g_M\to0$ on $B$ with respect to the weak $C^\infty$ topology.
\end{proof}

\begin{prop}\label{p: holonomy covering}
	The holonomy covering of any leaf $\im\hat\iota_{M,f}$ of $\widehat\FF_{*,\text{\rm imm}}(n)$ is $\hat\iota_{M,f}:M\to\im\hat\iota_{M,f}$.
\end{prop}

This proposition follows directly from the obvious version of \cite[Lemma~11.9]{AlvarezBarralCandel2016} for $\widehat\MM_{*,\text{\rm imm}}^\infty(n)$.

\section{Universality}\label{s: universality}

\begin{defn}\label{d: covering-cont} 
  	Let $X$ be a sequential Riemannian foliated space with complete leaves, and let $L_x$ denote the leaf through every $x\in X$, whose holonomy covering is denoted by $\widetilde L_x^{\text{\rm hol}}$. It is said that $X$ is \emph{covering-continuous} when there is a connected pointed covering $(\widetilde L_x,\tilde x)$ of $(L_x,x)$ for all $x\in X$ such that $[\widetilde L_{x_i},\tilde x_i]$ is $C^\infty$ convergent to $[\widetilde L_x,\tilde x]$ if $x_i\to x$ is a convergent sequence in $X$. When this condition is satisfied with $\widetilde L_x=\widetilde L_x^{\text{\rm hol}}$ for all $x\in X$, it is said that $X$ is \emph{holonomy-continuous}.
\end{defn}

\begin{rem}\label{r: covering-continuous}
  	Observe the following:
		\begin{enumerate}[(i)]
	
			\item\label{i: covering-determined} Covering-continuity and holonomy-continuity are weaker than covering-determination and holonomy-determination \cite[Definition~12.1]{AlvarezBarralCandel2016}, which were defined by using ``if and only if'' instead of ``if''.
				
  			\item\label{i: hereditary} The condition of being covering-continuous is hereditary (by saturated subspaces).
			\item\label{i: nets} Covering/holonomy-continuity/determination have obvious generalizations to arbitrary Riemannian foliated spaces by using nets instead of sequences.
		
		\end{enumerate}
\end{rem}

\begin{ex}\label{ex: covering-cont}
	The following simple examples clarify Definition~\ref{d: covering-cont}:
  		\begin{enumerate}[(i)]
			
  			\item The Reeb foliation on $S^3$ with the standard metric is covering-continuous, but it is not holonomy-continuous with any Riemannian metric. If the metric is modified around the compact leaf $T^2=S^1\times S^1$ so that the diffeomorphism $(x,y)\mapsto(y,x)$ of $T^2$ is not an isometry, then this foliation becomes non-covering-continuous.
			
  			\item The Riemannian foliated space of \cite[Example~2.5]{Lessa2015} is covering-determined but not holonomy-continuous. This example can be easily realized as a saturated subspace of a Riemannian foliated space where the holonomy coverings of the leaves are isometric to $\R$. So holonomy-continuity is not hereditary.
		
  			\item\label{i: widehat MM_*,imm^infty(n) is holonomy-cont} $\widehat\MM_{*,\text{\rm imm}}^\infty(n)$ is holonomy-continuous. However it is not holonomy-determined for $n\ge1$ by \cite[Remark~10-(iii)]{AlvarezBarralCandel2016}, since there are different points with isometric pointed holonomy covers of the corresponding pointed leaves. To see this, take any connected complete Riemannian $n$-manifold $M$, and some $x\in M$ and $f,f'\in C^\infty_{\text{\rm imm}}(M,\E)$ such that $f(x)\ne f'(x)$. Then $\hat\iota_{M,f}(x)\ne\hat\iota_{M,f'}(x)$, but $(M,x)$ is isometric to the holonomy covers of the pointed leaves $(\im\hat\iota_{M,f},\hat\iota_{M,f}(x))$ and $(\im\hat\iota_{M,f'},\hat\iota_{M,f'}(x))$.
			
		\end{enumerate}
\end{ex}

\begin{prop}[{Cf.\ \cite[Theorem~11.4.4]{CandelConlon2000-I}}]\label{p: C^infty embedding X to E}
	For any Polish $C^\infty$ foliated space $X$ with complete leaves, there is a $C^\infty$ embedding $X\to\E$.
\end{prop}

\begin{proof}
	This is an adaptation of the usual argument to show the existence of $C^\infty$ embeddings of $C^\infty$ manifolds in Euclidean spaces \cite[Theorem~1.3.4]{Hirsch1976}. Let $n=\dim X$ (as foliated space), and let $B_r=B_{\R^n}(0,r)$ and $\ol B_r=\ol B_{\R^n}(0,r)$ for each $r>0$. 
	
	\begin{claim}\label{cl: h}
		Let $Z$ be a Polish space, and consider the $C^\infty$ foliated structure on $U:=B_2\times Z$ with leaves $B_2\times\{*\}$. Let $V$ and $W$ be open subsets of $U$ such that $\ol V\subset W$ and $\ol W\subset B_1\times Z$. Then there is some $h\in C^\infty(U)$ such that $h=1$ on $V$ and $\supp h\subset W$.
	\end{claim}
	
	Since $\ol B_1$ is compact, it easily follows that each $z\in Z$ has an open neighborhood $P_z$ in $Z$ such that, for some open subsets $G_z,H_z\subset B_2$ with $\ol{G_z}\subset H_z$ and $\ol{H_z}\subset B_1$, we have $\ol V\cap(B_1\times P_z)\subset G_z\times P_z$ and $\ol{H_z}\times P_z\subset W$. Let $\{\lambda_i\}$ be a partition of unity of $Z$ subordinated to the open cover $\{\,P_z\mid z\in Z\,\}$; in particular, for every $i$, there is some $z_i\in Z$ so that $\supp\lambda_i\subset P_{z_i}$. Let $h_i\in C^\infty(B_2)$ such that $h_i=1$ on $G_{z_i}$ and $\supp h_i\subset H_{z_i}$. Then $h_i\lambda_i\in C^\infty(U)$, $h_i\lambda_i=\lambda_i$ on $G_{z_i}\times P_{z_i}$ and $\supp(h_i\lambda_i)\subset H_{z_i}\times P_{z_i}$. It follows that $h=\sum_ih_i\lambda_i$ satisfies the properties stated in Claim~\ref{cl: h}.
	
	Now, let $\UU$ be a countable collection of $C^\infty$ foliated charts $\phi_i:U_{2,i}\to B_2\times Z_i$ of $X$ such that the open sets $U_{1,i}:=\phi_i^{-1}(B_1\times Z_i)$ cover $X$. Using the paracompactness and regularity of $X$, a standard argument gives locally finite open covers, $\VV=\{V_i\}$ and $\WW=\{W_i\}$, with the same index set as $\UU$, such that $\ol{V_i}\subset W_i$ and $\ol{W_i}\subset U_{1,i}$. For each $i$, let $\E_i$ be a copy of $\E$. Take embeddings $\psi_i:Z_i\to\E_i$ \cite[Corollary~IX.9.2]{Dugundji1978}. Thus each composite
		\[
			\begin{CD}
				U_{2,i} @>{\phi_i}>> B_2\times Z_i @>{\id\times\psi_i}>> B_2\times\E_i\hookrightarrow\R^n\times\E_i=:\widetilde\E_i
			\end{CD}
		\]
	is a $C^\infty$ embedding with respect to the restriction of $\FF$, which will be denoted by $\tilde\phi_i$. By Claim~\ref{cl: h}, there are functions $h_i\in C^\infty(U_{2,i})$ such that $h_i=1$ on $V_i$ and $\supp h_i\subset W_i$. Then a $C^\infty$ embedding\footnote{The notation $\widehat\bigoplus_i\F_i$ is used for the Hilbert space direct sum of a family of Hilbert spaces $\F_i$; i.e., the Hilbert space completion of $\bigoplus_i\F_i$ with the scalar product $\langle(v_i),(w_i)\rangle=\sum_i\langle v_i,w_i\rangle$.} $f:X\to\widehat\bigoplus_i\widetilde\E_i\cong\E$ is defined by $f(x)=\sum_ah_a(x)\tilde\phi_{i_{k_a}}$.

\end{proof}

\begin{proof}[Proof of Theorem~\ref{t: universal}]
	The Polish Riemannian foliated space $\widehat\MM_{*,\text{\rm imm}}^\infty(n)$ has complete leaves and is holonomy-continuous (Example~\ref{ex: covering-cont}-\eqref{i: widehat MM_*,imm^infty(n) is holonomy-cont}). Thus any Polish Riemannian foliated subspace of $\widehat\MM_{*,\text{\rm imm}}^\infty(n)$ is also covering-continuous (Remark~\ref{r: covering-continuous}-\eqref{i: hereditary}).
	
	Let $X$ be any covering continuous Polish Riemannian foliated space with complete leaves. By Proposition~\ref{p: C^infty embedding X to E}, there is a $C^\infty$ embedding $f:X\to\E$. With the notation of Definition~\ref{d: covering-cont}, suppose that the covering-continuity of $X$ is satisfied with the connected pointed coverings $(\widetilde L_x,\tilde x)\to(L_x,x)$ ($x\in X$). Let $\hat\iota_{X,f}:X\to\widehat\MM_{*,\text{\rm imm}}^\infty(n)$ be defined by $\hat\iota_{X,f}(x)=[\widetilde L_x,\tilde f_x,\tilde x]$, where $\tilde f_x$ is the lift of $f|_{L_x}$ to $\widetilde L_x$. This map is well defined because the leaves of $X$ are complete. Moreover it is obviously foliated and continuous by the definitions of covering-continuity and the topology of $\widehat\MM_{*,\text{\rm imm}}^\infty(n)$. 
	
	To show that $\hat\iota_{X,f}$ is $C^\infty$, take a foliated chart $\Phi=(\chi,\Theta):\NN_2\to B\times\ZZ$ of $\widehat\FF_{*,\text{\rm imm}}^\infty(n)$ defined by any choice of $(V,e,\rho,\kappa,\sigma)$ as above. Let $U$ be the domain of a foliated chart of $X$ such that $\hat\iota_{X,f}(U)\subset\NN_2$. Then the composite
		\[
			\begin{CD}
				U @>{\hat\iota_{X,f}}>> \NN_2 @>{\chi}>> B
			\end{CD}
		\]
	is equal to $\Pi_V\circ(f-e)$, and therefore it is $C^\infty$.
	
	Finally, $\hat\iota_{X,f}$ is a $C^\infty$ embedding because the composite 
		\[
			\begin{CD}
				X @>{\hat\iota_{X,f}}>> \widehat\MM_{*,\text{\rm imm}}^\infty(n) @>{\ev}>> \E
			\end{CD}
		\]
	equals the $C^\infty$ embedding $f$.
\end{proof}

\section{Realization of manifolds of bounded geometry as leaves}
\label{s: bounded geometry}

\begin{prop}\label{p: bounded geometry}
	Let $M$ be any connected, complete Riemannian $n$-manifold of bounded geometry. Then there is a $C^\infty$ embedding $f:M\to\E$ such that $\widehat{\Cl}_\infty(\im\hat\iota_{M,f})$ is a compact subspace of $\widehat\MM_{*,\text{\rm imm}}^\infty(n)$.
\end{prop}

\begin{proof}
	Let $B_r=B_{\R^n}(0,r)$ for each $r>0$. By the bounded geometry of $M$, there is some $r>0$, smaller than the injectivity radius of $M$, such that the following properties hold:
		\begin{enumerate}[(i)]
		
			\item\label{i: g_ij} For the normal parametrizations $\kappa_x:B_r\to B_M(x,r)$ ($x\in M$), the corresponding metric coefficients, $g_{ij}$ and $g^{ij}$, as a family of $C^\infty$ functions on $B_r$ parametrized by $x$, $i$ and $j$, lie in a bounded subset of the Fr\'echet space $C^\infty_b(B_r)$ \cite[Theorem~A.1]{Schick1996}, \cite[Theorem 2.5]{Schick2001} (see also \cite[Proposition~2.4]{Roe1988I}, \cite{Eichhorn1991}). 
			
			\item\label{i: B_M(x_i,r)} There is some countable subset $\{\,x_i\mid i\in\N\,\}\subset M$ and some $c\in\N$ such that the family of balls $B_M(x_i,r/2)$ covers $M$, and $B_M(x,r)$ meets at most $c$ sets $B_M(x_i,r)$ for all $x\in M$ \cite[A1.2 and~A1.3]{Shubin1992}, \cite[Proposition~3.2]{Schick2001}. 
			
		\end{enumerate}
	Let $\kappa_i=\kappa_{x_i}$ for each $i$. 

	\begin{claim}\label{cl: coloring of the covering}
		There is a partition of $\N$ into finitely many sets, $I_1,\ldots, I_{c+1}$, such that $B_M(x_i,r)\cap B_M(x_j,r)=\emptyset$ for $i\in I_k$ and $j\in I_l$ with $k\ne l$.
	\end{claim}

	This claim follows by considering the graph $G$ whose set of vertices is $\N$, and such that there is a unique edge connecting two different vertices, $i$ and $j$, if and only if $B_M(x_i,r)\cap B_M(x_j,r)\ne\emptyset$. Since there are at most $c$ edges meeting at each vertex according to~\eqref{i: B_M(x_i,r)}, $G$ is $c+1$-colorable\footnote{This easily follows by induction, assigning to each $i$ a color different from the colors of the previous vertices that are neighbors of $i$, which is possible because there are at most $c$ of them (see \cite{Brooks1941}).}; i.e., there is a partition of $\N$ into subsets, $I_1,\dots,I_{c+1}$, such that there is no edge joining any pair of different vertices in any $I_k$.

	Let $S$ be an isometric copy in $\R^{n+1}$ of the standard $n$-dimensional sphere containing the origin $0$. Choose some spherically symmetric $C^\infty$ function $\rho\in C^\infty(\R^n)$ such that $\rho(x)=1$ if $|x|\le r/2$ and $\rho(x)=0$ if $|x|\ge r$. Take also some $C^\infty$ map $\tau\colon\R^n\to\R^{n+1}$ that restricts to a diffeomorphism $B_r\to S\sm\{0\}$ and maps $\R^n\sm B_r$ to $0$. Let $\tilde\rho_i$ be the extension by zero of $\rho\circ \kappa_i^{-1}$ to the whole of $M$, and let $\tilde\rho^k=\sum_{i\in I^k}\tilde\rho_i$. For each $k$, define $f^k:M\to\R^{n+2}$ by
		\[
			f^k(x)=
				\begin{cases}
					0 & \text{if $x\notin \bigcup_{i\in I^k}B_M(x_i,r)$}\\
					\left(\tilde\rho^k(x)/i,\tilde\rho^k(x)\cdot\tau\circ\kappa_i^{-1}(x)\right)
					& \text{if $x\in B_M(x_i,r)$ for some $i\in I^k$}\;.
				\end{cases}
		\]
	So $f^k\circ \kappa_i=(\rho/i,\rho\cdot\tau)$, obtaining that, for every multi-index $\alpha$, the function $|\partial_\alpha(f^k\circ \kappa_i)|$ is uniformly bounded over $B_r$ by a constant depending only on $|\alpha|$. Let $f=(f^1,\dots,f^{c+1}):M\to\R^{(c+1)(n+2)}$. We have $\sup_M|\nabla^mf|<\infty$ for each $m\in\N$ by~\eqref{i: g_ij}. Moreover $f^k\circ \kappa_i=(1/i,\tau)$ on $B_{r/2}$, obtaining that $f$ is a $C^\infty$ embedding, and $\inf_M|\bigwedge^n df|>0$ by~\eqref{i: g_ij}. By taking any isometric linear embedding of $\R^{(c+1)(n+2)}$ into $\E$, we can consider $\R^{(c+1)(n+2)}$-valued functions as $\E$-valued functions; in particular, this applies to $f$.
	
	\begin{claim}\label{cl: imm}
		$\widehat{\Cl}_\infty(\im\hat\iota_{M,f})\subset\widehat\MM_{*,\text{\rm imm}}^\infty(n)$.
	\end{claim}
	
	This claim is true because, for all $[N,h,y]\in\widehat{\Cl}_\infty(\im\hat\iota_{M,f})$, it is easy to see that $\inf_N|\bigwedge^ndh|\ge\inf_M|\bigwedge^ndf|>0$, obtaining that $h$ is an immersion.
	
	\begin{claim}\label{cl: compact}
		$\widehat{\Cl}_\infty(\im\hat\iota_{M,f})$ is compact.
	\end{claim}
	
	This assertion follows by showing that any sequence in $\im\hat\iota_{M,f}$ has a subsequence that is convergent in $\widehat\MM_*^\infty(n)$. Assume first that the sequence is of the form $[M,f,x_{i_p}]$ for some sequence of indices $i_p$. Since $\Cl_\infty(\im\iota_M)$ is compact in $\MM_*^\infty(n)$ by \cite[Theorem~12.3]{AlvarezBarralCandel2016}, we can suppose that $[M,x_{i_p}]$ converges to some point $[N,y]$ in $\MM_*^\infty(n)$. Take a sequence of compact domains $\Omega_q$ in $N$ such that $B_N(y,q+1)\subset\Omega_q$. For each $q$, there are pointed local embeddings $\phi_{q,p}:(N,y)\rightarrowtail(M,x_{i_p})$, for $p$ large enough, such that $\Omega_q\subset\dom\phi_{q,p}$ and $\phi_{q,p}^*g_M\to g_N$ on $\Omega_q$ with respect to the $C^\infty$ topology. Let $h_{q,p}=\phi_{q,p}^*f$ on $\Omega_q$. It is easy to see that, for all naturals $q$ and $m$, the sequence $\|h_{q,p}\|_{C^m,\Omega_q,g_N}$ is uniformly bounded. Hence the functions $h_{q,p}$ form a compact subset of $C^\infty(\Omega_q,\R^{(c+1)(n+2)})$ with the $C^\infty$ topology by \cite[Proposition~3.11]{AlvarezBarralCandel2016}. So some subsequence $h_{q,p(q,\ell)}$ is convergent to some $h_q\in C^\infty(\Omega_q,\R^{(c+1)(n+2)})$ with the $C^\infty$ topology. In fact, arguing inductively on $q$, it is easy to see that we can assume that each $h_{q+1,p(q+1,\ell)}$ is a subsequence of $h_{q,p(q,\ell)}$, and therefore $h_{q+1}$ extends $h_q$. Thus the functions $h_q$ can be combined to define a function $h\in C^\infty(M,\R^{(c+1)(n+2)})$. Take sequences of integers, $\ell_q\uparrow\infty$ and $m_q\uparrow\infty$, so that
		\[
			\|h-\phi_{q,p(q,\ell_q)}^*f\|_{C^{m_q},\Omega_q,g_N}
			=\|h_q-h_{q,p(q,\ell_q)}\|_{C^{m_q},\Omega_q,g_N}\to0\;.
		\]
	Then, considering $h$ as an $\E$-valued function, we get that $[M,f,x_{i_{p(q,\ell_q)}}]\to[N,h,y]$ in $\widehat\MM_*^\infty(n)$ as $q\to\infty$.
	
	Now take an arbitrary sequence $[M,f,x'_p]$ in $\im\hat\iota_{M,f}$. By~\eqref{i: B_M(x_i,r)}, there is a sequence of naturals, $i_p$, such that $d_M(x'_p,x_{i_p})<r/2$. By the above case in the proof, after taking a subsequence if necessary, we can assume that $[M,f,x_{i_p}]$ is convergent to some point $[N,h,y]$ in $\widehat\MM_*^\infty(n)$. Thus, given sequences, $m_j\uparrow\infty$ in $\N$, and $S_j\uparrow\infty$ and $s_j\downarrow0$ in $\R^+$, there is some sequence $p_j\uparrow\infty$ in $\N$ such that there exists some $(m_j,S_j+e^{s_j}r/2,\lambda_j,\epsilon_j)$-pointed local quasi-equivalence $\phi_j:(N,h,y)\rightarrowtail(M,f,x_{i_{p_j}})$ for some $\lambda_j\in[1,e^{s_j})$ and $\epsilon_j\in(0,s_j)$. Since $y'_j:=\phi_j^{-1}(x'_{p_j})\in B_N(y,e^{s_j}r/2)$, it follows that $\phi_j:(N,h,y'_j)\rightarrowtail(M,f,x'_{p_j})$ is an $(m_j,S_j,\lambda_j,\epsilon_j)$-pointed local quasi-equivalence, showing that $[M,f,x'_{p_j}]\in\widehat U^{m_j}_{S_j,s_j}(N,h,y'_j)$. On the other hand, since the sequence $y'_j$ is bounded in $N$, we can suppose that it is convergent to some $y'\in N$ by taking a subsequence if necessary. Hence $[N,h,y'_j]\to[N,h,y']$ in $\widehat\MM_*^\infty(n)$ by the continuity of $\hat\iota_{N,h}$. Hence there are sequences, $n_j\uparrow\infty$ in $\N$, and $T_j\uparrow\infty$ and $t_j\downarrow\infty$ in $\R^+$, such that $[N,h,y'_j]\in\widehat U^{n_j}_{e^{s_j}T_j,t_j}(N,h,y')$ for $j$ large enough. So
		\[
			[M,f,x'_{p_j}]\in\widehat U^{m_j}_{S_j,s_j}\circ\widehat U^{n_j}_{e^{s_j}T_j,t_j}(N,h,y')
			\subset\widehat U^{\min\{m_j,n_j\}}_{\min\{S_j,T_j\},s_j+t_j}(N,h,y')
		\]
	for $p$ large enough by Propositionn~\ref{p: C^infty uniformity in widehat MM_*(n)}-\eqref{i: circ}. This shows that $[M,f,x'_{p_j}]\to[N,h,y']$ in $\widehat\MM_*^\infty(n)$, completing the proof of Claim~\ref{cl: compact}.
\end{proof}	

\begin{proof}[Proof of Theorem~\ref{t: bounded geometry => leaf of a compact fol sp}]
	Given a  connected, complete Riemannian $n$-manifold $M$ of bounded geometry, by Proposition~\ref{p: bounded geometry}, and Theorems~\ref{t: C^infty convergence in widehat MM_*(n)} and~\ref{t: widehat FF_*,imm(n)}, $\widehat{\Cl}_\infty(\im\hat\iota_{M,f})$ is a compact Riemannian foliated subspace of $\widehat\MM_{*,\text{\rm imm}}^\infty(n)$. Moreover $\hat\iota_{M,f}:M\to\im\hat\iota_{M,f}$ is an isometry because $f$ is an embedding.
\end{proof}

\section{Open problems}

\begin{quest}\label{quest: trivial hol}
	In Theorem~\ref{t: bounded geometry => leaf of a compact fol sp}, is it possible to get the Riemannian foliated space so that its leaves have trivial holonomy?
\end{quest}

Question~\ref{quest: trivial hol} can be reduced to the following question, in the same way as Theorem~\ref{t: bounded geometry => leaf of a compact fol sp} follows from Proposition~\ref{p: bounded geometry}.

\begin{quest}\label{quest: aperiodic}
	In Proposition~\ref{p: bounded geometry}, is it possible to get $f$ such that moreover\footnote{According to the terminology for tilings, it could be said that $(M,f)$ is \emph{aperiodic} when this condition is satisfied. The same term could be also used for the corresponding property for graphs, in Question~\ref{quest: graph}.} $\Iso(N,h)=\{\id_M\}$ if $\im\hat\iota_{N,h}\subset\widehat{\Cl}_\infty(\im\hat\iota_{M,f})$?
\end{quest}
	
	In turn, Question~\ref{quest: aperiodic} can be reduced to the following graph version. Consider only connected graphs with a countable set of vertices, all of them with finite degree. These graphs are proper path metric spaces in a canonical way so that each edge is of length one. Thus they define a subspace $\GG_*$ of the Gromov space $\MM_*$ of pointed proper metric spaces. Decorate such graphs with maps of their vertex set to $\N$. This gives rise to a space $\widehat\GG_*$ of isomorphism classes of pointed decorated graphs, like in the case of $\widehat\MM_*^\infty(n)$. Let $\widehat{\Cl}$ denote the closure operator in $\widehat\GG_*$. For each decorated graph $(G,\alpha)$, let $\Iso(G,\alpha)$ denote its group of isomorphisms. There is a canonical map $\hat\iota_{G,\alpha}:G\to\widehat\GG_*$, like the above map $\hat\iota_{M,f}$. It is said that $G$ is of \emph{bounded geometry} if there is a uniform upper bound for the degree of its vertices.

\begin{quest}\label{quest: graph}
	For any graph $G$ of bounded geometry, does there exist a finite valued decoration $\alpha$ so that $\Iso(H,\beta)=\{\id\}$ for all decorated graph $(H,\beta)$ with $\im\hat\iota_{H,\beta}\subset\widehat{\Cl}(\im\hat\iota_{G,\alpha})$?
\end{quest}

There are aperiodic tilings of $\R$ (like the Fibonacci tiling), or elements of $\{0,1\}^\Z$, giving rise to examples of decorations of the Cayley graph of $\Z$ satisfying the condition of Question~\ref{quest: graph} (see e.g.\ \cite{PytheasFogg2002}). If Question~\ref{quest: graph} had an affirmative answer, then, in the proof of Proposition~\ref{p: bounded geometry}, we could take a finite valued decoration $\alpha$ of $G$ satisfying the condition of Question~\ref{quest: graph}, and modify the definition of $f$ so that
	\[
		f^k(x)=\left(\tilde\rho^k(x)\cdot(\alpha(i)+1/i),\tilde\rho^k(x)\cdot\tau\circ\kappa_i^{-1}(x)\right)
	\]
if $x\in B_M(x_i,r)$ for some $i\in I^k$. This would give affirmative answers to Questions~\ref{quest: aperiodic} and~\ref{quest: trivial hol}.

\bibliographystyle{amsplain}


\begin{thebibliography}{10}

\bibitem{AlvarezBarralCandel2016}
J.A. \'Alvarez~L\'opez, R.~Barral~Lij\'o, and A.~Candel, \emph{A universal
  {R}iemannian foliated space}, Topology Appl. \textbf{198} (2016), 47--85.
  \MR{3433188}

\bibitem{AttieHurder1996}
O.~Attie and S.~Hurder, \emph{Manifolds which cannot be leaves of foliations},
  Topology \textbf{35} (1996), 335--353. \MR{1380502 (96m:57037)}

\bibitem{Brooks1941}
R.L. Brooks, \emph{On colouring the nodes of a network}, Proc. Cambridge
  Philos. Soc. \textbf{37} (1941), 194--197. \MR{0012236 (6,281b)}

\bibitem{CandelConlon2000-I}
A.~Candel and L.~Conlon, \emph{Foliations. {I}}, Graduate Studies in
  Mathematics, vol.~23, American Mathematical Society, Providence, RI, 2000.
  \MR{1732868 (2002f:57058)}

\bibitem{CandelConlon2003-II}
\bysame, \emph{Foliations. {II}}, Graduate Studies in Mathematics, vol.~60,
  American Mathematical Society, Providence, RI, 2003. \MR{1994394
  (2004e:57034)}

\bibitem{CantwellConlon1987}
J.~Cantwell and L.~Conlon, \emph{Every surface is a leaf}, Topology \textbf{26}
  (1987), 265--285. \MR{899049 (89d:57039)}

\bibitem{ChaluleauPittet2001}
B.~Chaluleau and C.~Pittet, \emph{Exemples de vari\'et\'es riemanniennes
  homog\`enes qui ne sont pas quasi isom\'etriques \`a un groupe de type fini},
  C. R. Acad. Sci. Paris S\'er. I Math. \textbf{332} (2001), 593--595.
  \MR{1841890 (2002c:22013)}

\bibitem{Cheeger1970}
J.~Cheeger, \emph{Finiteness theorems for {R}iemannian manifolds}, Amer. J.
  Math. \textbf{92} (1970), 61--74. \MR{0263092 (41 \#7697)}

\bibitem{Dugundji1978}
J.~Dugundji, \emph{Topology}, Allyn and Bacon Inc., Boston, Mass., 1978,
  reprinting of the 1966 original, Allyn and Bacon Series in Advanced
  Mathematics. \MR{0478089 (57 \#17581)}

\bibitem{Eichhorn1991}
J.~Eichhorn, \emph{The boundedness of connection coefficients and their
  derivatives}, Math. Nachr. \textbf{152} (1991), 145--158. \MR{1121230
  (92k:53069)}

\bibitem{EskinFisherWhyte2012}
A.~Eskin, D.~Fisher, and K.~Whyte, \emph{Coarse differentiation of
  quasi-isometries {I}: {S}paces not quasi-isometric to {C}ayley graphs}, Ann.
  of Math. (2) \textbf{176} (2012), 221--260. \MR{2925383}

\bibitem{Ghys1985}
\'E. Ghys, \emph{Une vari\'et\'e qui n'est pas une feuille}, Topology
  \textbf{24} (1985), 67--73. \MR{790676 (87j:57014)}

\bibitem{Ghys2000}
{\'E}.~Ghys, \emph{Laminations par surfaces de {R}iemann}, Panoramas \&
  Synth\`eses \textbf{8} (2000), 49--95. \MR{1760843 (2001g:37068)}

\bibitem{Greene1978}
R.E. Greene, \emph{Complete metrics of bounded curvature on noncompact
  manifolds}, Arch. Math. (Basel) \textbf{31} (1978), 89--95. \MR{510080
  (81h:53035)}

\bibitem{Gromov1981}
M.~Gromov, \emph{Groups of polynomial growth and expanding maps. {Appendix by
  Jacques Tits}}, Inst. Hautes \'Etudes Sci. Publ. Math. (1981), no.~53,
  53--73. \MR{623534 (83b:53041)}

\bibitem{Gromov1999}
\bysame, \emph{Metric structures for {R}iemannian and non-{R}iemannian spaces},
  Progress in Mathematics, vol. 152, Birkh{\"a}user Boston Inc., Boston, MA,
  1999, Based on the 1981 French original [MR0682063 (85e:53051)], With
  appendices by M.~Katz, P.~Pansu and S.~Semmes, Translated from the French by
  Sean Michael Bates. \MR{1699320 (2000d:53065)}

\bibitem{GutierresHofmann2007}
G.~Gutierres and D.~Hofmann, \emph{Axioms for sequential convergence}, Appl.
  Categ. Structures \textbf{15} (2007), no.~5-6, 599--614. \MR{2365760
  (2009a:54021)}

\bibitem{Hirsch1976}
M.W. Hirsch, \emph{Differential topology}, Graduate Texts in Mathematics,
  no.~33, Springer-Verlag, New York-Heidelberg, 1976. \MR{0448362 (56 \#6669)}

\bibitem{InabaNishimoriTakamuraTsuchiya1985}
T.~Inaba, T.~Nishimori, M.~Takamura, and N.~Tsuchiya, \emph{Open manifolds
  which are nonrealizable as leaves}, Kodai Math. J. \textbf{8} (1985),
  112--119. \MR{776712 (86f:57024)}

\bibitem{Kechris1995}
A.S. Kechris, \emph{Classical descriptive set theory}, Graduate Texts in
  Mathematics, vol. 156, Springer-Verlag, New York, 1995. \MR{1321597
  (96e:03057)}

\bibitem{Koutnik1985}
V.~Koutn\'{\i}k, \emph{Closure and topological sequential convergence},
  Convergence structures, Proc. Conf., Bechyn\v{e}/Czech. 1984, Math. Res.,
  vol.~24, Akademie-Verlag, Berlin, 1985, pp.~199--204. \MR{835486 (87e:54005)}

\bibitem{Lessa2015}
P.~Lessa, \emph{{Reeb} stability and the {Gromov-Hausdorff} limits of leaves in
  compact foliations}, Asian J. Math. \textbf{19} (2015), 433--464.
  \MR{3361278}

\bibitem{MooreSchochet1988}
C.C. Moore and C.~Schochet, \emph{Global analysis on foliated spaces},
  Mathematical Sciences Research Institute Publications, vol.~9,
  Springer-Verlag, New York, 1988, With appendices by S. Hurder, Moore,
  Schochet and Robert J. Zimmer. \MR{0918974 (89h:58184)}

\bibitem{Petersen1998}
P.~Petersen, \emph{Riemannian geometry}, Graduate Texts in Mathematics, vol.
  171, Springer-Verlag, New York, 1998. \MR{1480173 (98m:53001)}

\bibitem{PytheasFogg2002}
N.~Pytheas~Fogg, \emph{Substitutions in dynamics, arithmetics and
  combinatorics}, Lecture Notes in Math., vol. 1794, Springer, Berlin, 2002,
  Edited by V.~Berth{\'e}, S.~Ferenczi, C.~Mauduit and A.~Siegel. \MR{1970385
  (2004c:37005)}

\bibitem{Roe1988I}
J.~Roe, \emph{An index theorem on open manifolds. {I}}, J. Differential Geom.
  \textbf{27} (1988), 87--113. \MR{918459 (89a:58102)}

\bibitem{Sasaki1958}
S.~Sasaki, \emph{On the differentiable geometry of tangent bundles of
  {Riemannian} manifolds}, T\^ohoku Math. J. (2) \textbf{10} (1958), no.~3,
  211--368. \MR{0112152 (22 \#3007)}

\bibitem{Schick1996}
T.~Schick, \emph{Analysis on $\partial$-manifolds of bounded geometry,
  {Hodge-De~Rham} isomorphism and {$L^2$}-index theorem}, Ph.D. thesis,
  Johannes Gutenberg Universit\"at Mainz, Mainz, 1996.

\bibitem{Schick2001}
\bysame, \emph{Manifolds with boundary and of bounded geometry}, Math. Nachr.
  \textbf{223} (2001), 103--120. \MR{1817852 (2002g:53056)}

\bibitem{Schweitzer1995}
P.A. Schweitzer, \emph{Surfaces not quasi-isometric to leaves of foliations of
  compact $3$-manifolds}, Analysis and geometry in foliated manifolds,
  Proceedings of the VII International Colloquium on Differential Geometry,
  Santiago de Compostela, Spain, July 26--30, 1994, World Sci. Publ.,
  Singapore, 1995, pp.~223--238. \MR{1414206 (98f:57031)}

\bibitem{Schweitzer2011}
\bysame, \emph{Riemannian manifolds not quasi-isometric to leaves in
  codimension one foliations}, Ann. Inst. Fourier (Grenoble) \textbf{61}
  (2011), 1599--1631. \MR{2951506}

\bibitem{SchweitzerSouza2013}
P.A. Schweitzer and F.S. Souza, \emph{Manifolds that are not leaves of
  codimension one foliations}, Int. J. Math. \textbf{24} (2013), 14 pp.
  \MR{3163616}

\bibitem{Shubin1992}
M.A. Shubin, \emph{Spectral theory of elliptic operators on noncompact
  manifolds}, Ast\'erisque \textbf{207} (1992), 35--108, M{\'e}thodes
  semi-classiques, Vol. 1 (Nantes, 1991). \MR{1205177 (94h:58175)}

\bibitem{Sondow1975}
J.D. Sondow, \emph{When is a manifold a leaf of some foliation?}, Bull. Amer.
  Math. Soc. \textbf{81} (1975), 622--625. \MR{0365591 (51 \#1843)}

\bibitem{Souza2011}
F.S. Souza, \emph{Non-leaves of some foliations}, Ph.D. thesis, Pontif\'icia
  Universidade Cat\'olica do Rio de Janeiro, 2011.

\bibitem{Willard1970}
S.~Willard, \emph{General topology}, Addison-Wesley Publishing Co., Reading,
  Mass.-London-Don Mills, Ont., 1970. \MR{0264581 (41 \#9173)}

\bibitem{Zeghib1994}
A.~Zeghib, \emph{An example of a $2$-dimensional no leaf}, Geometric study of
  foliations ({T}okyo, 1993), World Sci. Publ., River Edge, NJ, 1994,
  pp.~475--477. \MR{1363743 (96j:57033)}

\end{thebibliography}


\providecommand{\bysame}{\leavevmode\hbox to3em{\hrulefill}\thinspace}
\providecommand{\MR}{\relax\ifhmode\unskip\space\fi MR }
\providecommand{\MRhref}[2]{%
  \href{http://www.ams.org/mathscinet-getitem?mr=#1}{#2}
}
\providecommand{\href}[2]{#2}

\end{document}